\documentclass[a4paper,leqno,12pt]{amsart}
\usepackage[utf8]{inputenc}
\usepackage[english]{babel}
\usepackage{amsthm}
\usepackage{relsize}
\usepackage{amsmath, amssymb, mathrsfs}
\usepackage{hyperref}
\usepackage{cleveref}
\usepackage[all]{xy}
\usepackage{xypic}
\usepackage{tikz}
\usetikzlibrary{patterns}
\usetikzlibrary{shapes}
\xyoption{poly}
\usepackage{graphicx}
\usepackage{aliascnt}
\theoremstyle{plain}
\swapnumbers
    \newtheorem{theorem}{Theorem}[section]
    \newtheorem{prop}[theorem]{Proposition}
    \newtheorem{lemma}[theorem]{Lemma}
    
    \newtheorem{corollary}[theorem]{Corollary}

\newtheorem{subsec}[theorem]{}
    \newtheorem*{thma}{Theorem A}
    \newtheorem*{thmb}{Theorem B}

\theoremstyle{definition}
\newtheorem{defn}[theorem]{Definition}
\newtheorem{construction}[theorem]{Construction}
\newtheorem{example}[theorem]{Example}
    
    \newtheorem{examples}[theorem]{Examples}

   \theoremstyle{remark}
            
        \newtheorem{remark}[theorem]{Remark}

\renewcommand{\thefigure}{\arabic{section}.\arabic{theorem}}

\begin{document}
%
%
\newcommand{\hsp}{\hspace*{7 mm}}
\newcommand{\hs}{\hspace*{3 mm}}
\newcommand{\hsk}{\hspace*{-5 mm}}
\newcommand{\hsm}{\hspace*{3 mm}}
\newcommand{\hsn}{\hspace*{3 mm}}
\newcommand{\vsn}{\vspace*{2 mm}}
\newcommand{\vsneg}{\vspace{-1 mm}}
\newcommand{\vs}{\vspace{7 mm}}
\newcommand{\vsm}{\vspace{4 mm}}
\newenvironment{myeq}[1][]
{\stepcounter{theorem}\begin{equation}\tag{\thetheorem}{#1}}
{\end{equation}}
\newcommand{\mydiagram}[2][]
{\stepcounter{theorem}\begin{equation}
     \tag{\thetheorem}{#1}\vcenter{\xymatrix{#2}}\end{equation}}
\newcommand{\mysdiag}[2][]
{\stepcounter{theorem}\begin{equation}
     \tag{\thetheorem}{#1}\vcenter{\xymatrix@C=0.2em@R=0.8em{#2}}\end{equation}}
\newcommand{\mytdiag}[2][]
{\stepcounter{theorem}\begin{equation}
     \tag{\thetheorem}{#1}\vcenter{\xymatrix@C=0.8em@R=0.8em{#2}}\end{equation}}
\newcommand{\myvdiag}[2][]
{\stepcounter{theorem}\begin{equation}
     \tag{\thetheorem}{#1}\vcenter{\xymatrix@C=1.3em@R=0.8em{#2}}\end{equation}}
\newcommand{\myudiag}[2][]
{\stepcounter{theorem}\begin{equation}
     \tag{\thetheorem}{#1}\vcenter{\xymatrix@C=3.8em@R=0.8em{#2}}\end{equation}}
\newcommand{\myzdiag}[2][]
{\stepcounter{theorem}\begin{equation}
     \tag{\thetheorem}{#1}\vcenter{\xymatrix@C=2.3em@R=0.8em{#2}}\end{equation}}
\newcommand{\mytikz}[2][]
{\stepcounter{theorem}\begin{equation}    \tag{\thetheorem}{#1}\vcenter{\begin{center}\begin{tikzpicture}{#2}\end{tikzpicture}\end{center}}\end{equation}}
\renewcommand{\thefigure}{\arabic{section}.\arabic{theorem}}
\newenvironment{mysubsection}[2][]
{\begin{subsec}\begin{upshape}\begin{bfseries}{#2.}
\end{bfseries}{#1}}
{\end{upshape}\end{subsec}}
\newcommand{\sect}{\setcounter{theorem}{0}\section}
%
%
\newcommand{\cB}{\mathcal{B}}
\newcommand{\cC}{\mathcal{C}}
\newcommand{\mC}{\mathscr{C}}
\newcommand{\cDD}{\mathcal{D}}
\newcommand{\mD}{\mathscr{D}}
\newcommand{\cF}{F}
\newcommand{\cG}{G}
\newcommand{\ccH}{H}
\newcommand{\cK}{\mathcal{K}}
\newcommand{\ccL}{M}
\newcommand{\cM}{\mathcal{M}}
\newcommand{\mM}{\mathscr{M}}
\newcommand{\cP}{\mathcal{P}}
\newcommand{\mP}{\mathbf{P}}
\newcommand{\mT}{\mathcal{T}}
\newcommand{\bmT}{\bar{\mT}}
\newcommand{\mX}{\mathscr{X}}
\newcommand{\mY}{\mathscr{Y}}
\newcommand{\bHom}{\mathbf{Hom}}
\newcommand{\bL}{\mathbf{L}}
\newcommand{\bR}{\mathbf{R}}
\newcommand{\bX}{\mathbf{X}}
\newcommand{\bY}{\mathbf{Y}}
\newcommand{\ppp}[1]{[#1]}
\newcommand{\lra}[1]{\langle{#1}\rangle}
\newcommand{\hra}{\hookrightarrow}
%
%
\newcommand{\bk}{\mathbf{k}}
\newcommand{\pbk}{\ppp{\bk}}
\newcommand{\bl}{\mathbf{j}}
\newcommand{\pbl}{\ppp{\bl}}
\newcommand{\bm}{\mathbf{m}}
\newcommand{\pbm}{\ppp{\bm}}
\newcommand{\bn}{\mathbf{n}}
\newcommand{\pbn}{\ppp{\bn}}
\newcommand{\bnp}{\bn+\bo}
\newcommand{\br}{\mathbf{r}}
\newcommand{\bo}{\mathbf{1}}
\newcommand{\bbo}{\lra{\bo}}
\newcommand{\pbo}{\ppp{\bo}}
\newcommand{\btw}{\mathbf{2}}
\newcommand{\bth}{\mathbf{3}}
%
%
\newcommand{\xd}{x\sb{\bullet}}
\newcommand{\xu}{x\sp{\bullet}}
\newcommand{\zd}{z\sb{\bullet}}
\newcommand{\zu}{z\sp{\bullet}}
\newcommand{\wzu}{\widetilde{z}\sp{\bullet}}
%
%
\newcommand{\NN}{\mathbb{N}}
\newcommand{\RR}{\mathbb{R}}
\newcommand{\bottom}{\operatorname{bottom}}
\newcommand{\colim}{\operatorname{colim}}
\newcommand{\homm}{\operatorname{hom}}
\newcommand{\Hom}{\operatorname{Hom}}
\newcommand{\Id}{\operatorname{Id}}
\newcommand{\Image}{\operatorname{Im}}
\newcommand{\Iso}{\operatorname{Iso}}
\newcommand{\Ker}{\operatorname{Ker}}
\newcommand{\Map}{\operatorname{Map}}
\newcommand{\Ob}{\operatorname{Ob}}
\newcommand{\op}{\sp{\operatorname{op}}}
\newcommand{\Tot}{\operatorname{Tot}}
\newcommand{\tr}{\operatorname{tr}}
\newcommand{\sk}{\operatorname{sk}}
\newcommand{\Ho}{\operatorname{Ho}}
\newcommand{\Cat}{\mbox{\sf Cat}}
\newcommand{\sCat}{\mbox{\sf sCat}}
\newcommand{\DK}{\mbox{\sf DK}}
\newcommand{\Graph}{\mbox{\sf Graph}}
\newcommand{\SeCat}{\mbox{\sf SeCat}}
\newcommand{\Set}{\mbox{\sf Set}}
\newcommand{\sSet}{\mbox{\sf sSet}}
\newcommand{\stSet}{\mbox{\sf s}^{2}\mbox{\sf Set}}
\newcommand{\sSa}{\sSet\sb{\ast}}
\newcommand{\Sing}{\mbox{\sf Sing}}
\newcommand{\Kan}{\mbox{\sf Kan}}
\newcommand{\fd}{\mathfrak{d}}
\newcommand{\fB}{\mathfrak{B}}
\newcommand{\fC}{\mathfrak{C}}
\newcommand{\fU}{\mathfrak{U}}
\newcommand{\bDel}{\bar{\Delta}}
\newcommand{\pDel}{\sb{+}\Delta}
\newcommand{\pbDel}{\sb{+}\bDel}

%
\title{Spectral Sequences in $(\infty, 1)$-categories}
%
%
\author[D.~Blanc]{David Blanc}
\address{Department of Mathematics\\ University of Haifa\\ 3498838 Haifa\\ Israel}
\email{blanc@math.haifa.ac.il}
\author[N.J.~Meadows]{Nicholas J.\ Meadows}
\address{Department of Mathematics\\ University of Haifa\\ 3498838 Haifa\\ Israel}
\email{njmead81118@gmail.com}
\date{\today}

\begin{abstract}
  We explain how to set up the homotopy spectral sequence of a (co)simplicial
  object in an $\infty$-category, with an emphasis on how to construct the
  differentials in a model-invariant manner.
\end{abstract}

\maketitle

\setcounter{section}{0}

%
%
\sect{Introduction}
\label{cint}

Spectral sequences, first discovered by Leray in the 1940's, are a basic
computational tool of algebraic topology, homological algebra, group theory, and many related fields. In Massey's description using exact couples (see \cite{MasE}),
they appear to be purely algebraic devices, but in fact there is always a
homotopical (or at least homological) aspect to the differentials.

The exact couple is ordinarily derived from a tower of (co)fibrations in a model category $\cM$ (e.g., a filtered chain complex), and if $\cM$ is stable (see \cite[Ch.\ 7]{HovM}), the whole spectral sequence, including the differentials, can be described in terms of the triangulated structure (see \cite{CFranH,ShipA} and compare \cite{JCohDS}).  Thus setting up a spectral sequence in a stable $\infty$-category is straightforward: all that is needed
is a sequence of morphisms and the existence of suitable (co)limits
(see \cite[Section 1.3]{Lurie2}) \ - \ although even in this case, describing the differentials explicitly requires some work (see \S \ref{sdsic}).

In this paper we study unstable spectral sequences, and specifically the homotopy spectral sequence of a simplicial or cosimplicial object in a general $\infty$-category, as originally defined in \cite{BFrieH} (or \cite{QuiS}) and \cite{BK}, respectively.

The classical unstable Adams spectral sequences of \cite{BCKQRSclM,BKanS}, the generalized version of \cite{BCMillU}, the Eilenberg-Moore spectral sequence (see \cite{RectS}), and Anderson's generalization of the latter in \cite{AndG}, are among the many examples of the cosimplicial version. Segal's homology spectral sequence (see \cite{SegC}), the Serre spectral sequence (see \cite{DreS}), and the van Kampen spectral sequence (see \cite{StoV}), are examples of the simplicial version. Note also that each differential in the stable Adams spectral sequence, for a finite complex $\bY$, agrees with the unstable version for $\Sigma^{N}\bY$ if $N$ is large enough, so our results apply there too.

Even though both types of spectral sequences are derived from an appropriate tower of fibrations, in the unstable case this is not useful in analyzing the differentials in terms of the original (co)simplicial object. For example, one can use a cosimplicial resolution of $\bY$ to identify the differentials in the unstable Adams spectral sequence for $\bY$ as certain higher cohomology operations, though this is not apparent if one considers the associated $\Tot$ tower (see \cite{BBSenH}, and compare \cite{BJTurHH} for a simplicial analogue).

Our overall purpose is to explain how the differentials in the two
kinds of unstable homotopy spectral sequences can be described
explicitly in terms of maps induced by the universal property of
certain (homotopy) colimits. This presents the differentials as
homotopy invariants \ -- \ of the (co)simplicial object $\xd$ (or $\xu$),
and of the $\infty$-category $X$ in which it lies. These in turn
can be characterized in a manner independent of the particular
model of $\infty$-categories we use.

Note that we think of the differentials here as `relations', in the sense of
\cite[\S 3]{BKaSQ} \ -- \ that is, as a multi-valued function defined on representatives in the $E^{1}$-page. We observe that from this point of view there are also versions of the theory that work without assuming (abelian) group structure (cf.\ \cite{BousH,BKanS,BFrieH}), but we shall not consider them here.

To understand our approach to the problem, assume we are given a simplicial object
$\xd$ in a pointed $\infty$-category $X$ with enough (co)limits, and a
cogroup object $y$ in $X$. If $X$ is a simplicially enriched category, say, then
$\xd$ need only be a homotopy coherent diagram \ - \ that is, a simplicial functor out of the Dwyer-Kan cofibrant replacement $\DK(\Delta\op)$ for the usual indexing category
$\Delta\op$ (see \S \ref{rmk1.7}).

We then construct the homotopy spectral sequence for $\Map_{X}(y, \xd)$ by means of the spiral tower of fibrations, as in \cite{Dwyer-Kan-Stover} (see \S \ref{csssec} below), with
$$
E^{2}\sb{n,p}~=~\pi_{n}\pi_{p}\Map_{X}(y, \xd)~\Longrightarrow~
\pi_{p+n}\Map_{X}(y, \colim \xd)~.
$$

The differentials in the spectral sequence can be interpreted as obstructions to
lifting a certain diagram from the homotopy category of $X$:

\begin{thma}
  Let $\lra{f}\in E^{1}\sb{n,p}=\pi_{0}\Map_{X}(\Sigma^{p}y,x_{n})$ in the spectral sequence above  be represented by $f:\Sigma^{p}y \rightarrow x_{n}$. Then $\lra{f}$ survives to $E^{r}\sb{n,p}$
for $r\geq 2$ if and only if we can find a homotopy coherent diagram $D:\DK(J\sb{r-1})\rightarrow X$ of the form
\mydiagram[\label{diagnk}]{
\Sigma^{p}y  \ar@/^1pc/[rr]^{0}_{\vdots}  \ar@/_1pc/[rr]_{0} \ar[dd]  && 0  \ar@/^1pc/[rr]_{\vdots}  \ar@/_1pc/[rr] \ar[dd] && 0 \ar[dd]   & \cdots \cdots & 0 \ar[dd] \\
 && && && \\
x_{n} \ar@/^1pc/[rr]^{d_{0}}_{\vdots}  \ar@/_1pc/[rr]_{d_{n}} && x_{n-1}  \ar@/^1pc/[rr]^{d_{0}}_{\vdots}  \ar@/_1pc/[rr]_{d_{n-1}} && x_{n-2} & \cdots \cdots & x_{n-r+1}
}
\end{thma}

See Theorem \ref{thm3.5}.

Here $J\sb{r-1}=\pbDel\op_{n-r+1, n}
\times\bbo$ is the indexing category for the diagram
\eqref{diagnk} (see \S \ref{snac} below). Theorem A follows from:

\begin{thmb}
Suppose that we have a diagram $D:\DK(J\sb{r-1}) \rightarrow X$ as in \eqref{diagnk}. Then the differential $\fd^{r}\lra{f}$ is represented by
an element $\alpha \in \pi_{p+r}\Map_{X}(y, x_{n-r})$, constructed by the universal property of a certain colimit derived from the diagram $D$. In particular, $\alpha=0$ if and only if $D$ extends to a diagram $\DK(J\sb{r}) \rightarrow X$ as above.
\end{thmb}

See Theorem \ref{thm4.7} and Corollary \ref{cor4.8}.

By Remark \ref{rmk4.9}, we thus obtain a model-independent (and in particular, homotopy invariant) description of the differentials in the homotopy spectral sequence of a
simplicial object in an $\infty$-category.

\begin{mysubsection}{Convergence}
In this paper we are concerned only with the description of the differentials in the
spectral sequence of a (co)simplicial object, and will not discuss its convergence.
It should be noted, however, that if $y$ is a compact object in an $\infty$-category $X$
in the sense of \cite[\S 5.3.4]{Lurie} (so that $\Map_{X}(y, -)$ commutes with filtered
colimits), the spectral sequence associated to $y$ and a given
simplicial object $\xd$ in $X$ (see \S \ref{con2.6}) can be identified with the usual
spectral sequence of a simplicial space by Theorem \ref{thm3.2} below, and hence always
converges strongly. 

On the other hand, the spectral sequence of a cosimplicial space does not always converge
(see \cite[IX,\S 4]{BK} or \cite[VI]{GJ2}, especially Lemma 2.20), so the spectral
sequence of Section \ref{cdhssco} below need not converge in general.
\end{mysubsection}

\begin{mysubsection}{Notation and conventions}
\label{snac}
The category of sets is denoted by $\Set$, that of small categories by $\Cat$, and that of directed graphs by $\Graph$.

The category of non-empty finite ordered sets is denoted by $\Delta$, while that of all
finite ordered sets is denoted by $\bDel$. We denote the ordered
set $(0<1<\dotsc<n)$ by $\pbn$ (with $\ppp{-\bo}$ denoting the empty set in $\bDel$, by analogy). At times it will be convenient to
denote the one-arrow category (isomorphic to $\pbo$) by $\bbo$.
A \emph{cosimplicial} object in a category $\cC$ is thus a functor $\xu\colon\Delta\to\cC$,
a \emph{simplicial} object in $\cC$ is a functor $\xd\colon\Delta\op\to\cC$, and an
\emph{augmented} simplicial object in $\cC$ is a functor $\bDel\op\to\cC$.
As usual, we write $x_{n}$ for $\xd(\pbn)$ and $x^{n}$ for $\xu(\pbn)$.

The full subcategory of $\bDel$ generated by the monomorphisms
is denoted by $\pbDel$, so a functor $\pbDel\op\to\cC$ is a \emph{restricted augmented}
simplicial object (without degeneracies). Similarly, given integers $0\leq k\leq n$,
we denote by $\bDel_{k, n}$ the full subcategory of $\bDel$ whose
objects $\pbm$ satisfy  $k \le m \le n$, and by $\pbDel_{n}$
(or $\pbDel_{-1,n}$) the full subcategory of $\pbDel$ with $m \le n$.

The category of simplicial sets is denoted by $\sSet$, and that of pointed
simplicial sets by $\sSa$; an object in either is \emph{fibrant} if it is a
Kan complex (see \cite[Chapter I]{GJ2}), and $\Kan$ denotes the full subcategory of fibrant objects in $\sSet$.

A small category enriched in $\sSet$ will be called a \emph{simplicial category} \ -- \
equivalently, this can be thought of as a simplicial object in $\Cat$ for which
the face and degeneracy functors are the identity on objects. The category of
small simplicial categories will be denoted by $\sCat$, and we shall denote them by
$\mX$, $\mY$, and so on. For $\mX\in\sCat$, we write $\mX(a, b)$ 
for the simplicial set of morphisms from $a$ to $b$ (called the simplicial \emph{$\bHom$-set}).
\end{mysubsection}

\begin{mysubsection}{Outline}
\label{sorg}
Section \ref{cic} provides some background on $\infty$-categories
and in particular their axiomatic characterization. Section
\ref{cssso} gives an abstract description of the spectral sequence
of a simplicial object in such an $\infty$-category, mimicking the
construction in a simplicially enriched category.
In preparation for our main result, Section \ref{cdvho} analyzes the
differentials in the spectral sequence of a (strict) simplicial space,
and Section \ref{cdkrd} describes the Dwyer-Kan resolution of the restricted simplex
category $\pbDel$.
Section \ref{cdsscso} is the technical heart of the paper, describing the
differentials in the spectral sequence of a coherent simplicial object in
a simplicial category. In Section \ref{cqcic} we provide some additional
background on the quasi-category model of $\infty$-categories,
used in Section \ref{cdssqc} to explain the form taken by
differentials for the spectral sequence of a simplicial object in
a quasi-category.  Finally, in Section \ref{cdhssco} we briefly
address the dual case of the homotopy spectral sequence of a
cosimplicial object in an $\infty$-category.

\end{mysubsection}

\begin{mysubsection}{Acknowledgements}
  The research of first author was partially supported by Israel Science Foundation
  grant 770/16, and that of the second author by a Zuckerman Postdoctoral Fellowship.
\end{mysubsection}

%
%
\sect{Background on $\infty$-Categories}
\label{cic}

We begin with a rough axiomatization of the concept of $\infty$-categories, which covers most of the existing models such as simplicial categories (\cite{Bergner1}), quasi-categories (\cite{Lurie}), Segal categories (\cite{SimpsonBook}) and complete Segal spaces (\cite{Rezk}).

\begin{defn}\label{def1.1}
A \emph{model of} $\infty$-\emph{category theory} consists of the following data:

\begin{enumerate}
\renewcommand{\labelenumi}{(\arabic{enumi})}
\item A category $\cC$ and a full subcategory $\cC_{0}$ containing the initial object. The objects of $\cC_{0}$ are called $\infty$-\emph{categories}, and the morphisms are called \emph{functors of $\infty$-categories}.
\item A pair of adjoint functors $\cP\colon\cC_{0} \rightleftarrows \Cat\colon\cB$, with $\cB$ fully faithful. The left adjoint is the \emph{homotopy category} functor and the right adjoint is the \emph{nerve}.
\item A functor $\Ob\colon\cC\rightarrow \Set$, that assigns an $\infty$-category its set of objects.
\item If $X$ is an $\infty$-category, for each $x, y \in \Ob(X)$, we have a Kan complex of maps $\Map_{X}(x, y)$, with $\cP(x, y) = \pi_{0}\Map_{X}(x, y)$.
\item For each $x, y, z \in\Ob(X)$, we have a composition operation
$$
c_{x, y, z}\colon\Map_{X}(x, y) \times \Map_{X}(y, z) \rightarrow \Map_{X}(x, z),
$$
well defined and associative up to homotopy.
\item For each $x\in \Ob(X)$ we have an \emph{identity morphism}
  $\ast \xrightarrow{I_{x}} \Map_{X}(x, x) $, such that for each $y \in \Ob(X)$, both maps
$(\Map_{X}(y, x)\times\ast\xrightarrow{c_{y, x, x}\circ (\Id \times I_{x})} \Map_{X}(y, x))$
and $(\ast \times \Map_{X}(x, y)) \xrightarrow{c_{x, x, y} \circ ( I_{x} \times\Id)} \Map_{X}(x, y)$ are homotopic to the identity.
\item For each functor $F\colon X \rightarrow Y$, and $x, y \in X$, we have an induced map
$$
\Map_{X}(x, y) \rightarrow \Map_{Y}(F(x), F(y))
$$
which respects the  composition operation up to homotopy.
\end{enumerate}
\end{defn}

\begin{defn}\label{def1.2}
We call the vertices of $\Map_{X}(x, y)$  the \emph{morphisms} (or \emph{maps}) between $x$ and $y$. A $1$-simplex $a\in\Map_{X}(x, y)\sb{1}$ is called a \emph{homotopy} between $d_{0}(a)$ and $d_{1}(b)$, and we say that $d_{0}(a) \sim d_{1}(a)$ are \emph{homotopic}.

 A $2$-simplex $b$ of $\Map_{X}(x, y)$ is thought of as a homotopy between homotopies $d_{2}(b) \circ d_{0}(b) \sim d_{1}(b)$. More generally, the higher dimensional simplices can be thought of as higher homotopies relating the homotopies expressed by its faces one level below.
\end{defn}

\begin{defn}\label{def1.3}
We say that a functor $F\colon X \rightarrow Y$ of $\infty$-categories is an \emph{equivalence} if it induces weak equivalences of mapping spaces and  it induces an equivalence of homotopy categories $\cP(X) \rightarrow \cP(Y)$.
\end{defn}

\begin{mysubsection}{Simplicial categories}\label{exam1.5}
  The first model of $\infty$-category theory is the category of $\cC=\sCat$ of (small) simplicial categories. Here $\Ob\colon\cC\to\Set$ is the usual object functor, $\Map_{\mX}(a, b)=\mX(a, b)$ (see \S \ref{snac}), with the usual composition. Note that by definition of the morphisms in $\sCat$, a functor $F\colon\mX \rightarrow\mY$ induces a map of simplicial sets $\mX(a, b) \rightarrow\mY(a, b)$.

The nerve functor regards an ordinary category as a simplicial category with discrete simplicial $\bHom$-sets. The homotopy category $\cP(\mX)$ has object set $\Ob(\mX)$ and $\cP(\mX)(a, b):=\pi_{0}\Map_{\mX}(a, b)$.

Bergner showed that there is a model structure on the category of simplicial categories in which the weak equivalences are the \emph{Dwyer-Kan equivalences}: that is, maps
$f\colon\mX \rightarrow \mY$ inducing weak equivalences $\mX(a, b) \rightarrow \mY(a, b)$
in $\sSet$ for each $a,b\in\Ob\mX$, with $\cP(\mX) \rightarrow \cP(\mY)$ essentially surjective.

The fibrations are maps that induce Kan fibrations on simplicial $\bHom$-sets and induce isofibrations $\cP(\mX) \rightarrow \cP(\mY)$ (see \cite{Bergner1}),
and $\cC_{0}$ is the full subcategory of fibrant objects (that is, simplicial categories with fibrant simplicial $\bHom$-sets). The weak equivalences correspond to equivalences of $\infty$-categories in the sense of \S \ref{def1.3}. Thus, we can think of the above model structure as a model for $\infty$-category theory in the sense of \S \ref{def1.1}.
\end{mysubsection}

\begin{defn}\label{dfgf}
  The forgetful functor $U\colon\Cat\rightarrow \Graph$ from the category of small categories to the category of directed graphs has a left adjoint $F\colon\Graph\rightarrow \Cat$, the \emph{free category functor}. This pair of adjoint functors defines a comonad $FU\colon\Cat \rightarrow \Cat$ and hence a simplicial resolution
$\DK(C) \rightarrow C$ with $\DK(C)_{n} = (FU)^{n+1}C$. We call this the \emph{free resolution} of $C$.

If we think of a simplicially enriched category $\mX$ as a simplicial object in $\Cat$
(see \S \ref{snac}), we may define the \emph{Dwyer-Kan} resolution of $\mX$ by
$\DK(\mX)_{n} := \DK(\mX_{n})$. This serves as a functorial cofibrant replacement of $\mX$ in the Bergner model structure.
\end{defn}

\begin{remark}\label{rmk1.7}
The classical theory of homotopy coherence (see \cite{BVogHI} or \cite[Ch.\ IX]{GJ2}) was concerned with the extra data needed to lift a homotopy commutative diagram
  $I \rightarrow \Ho(\cM)$, in a simplicial model category $\cM$ to a strict diagram. It turns out that to obtain such a lift, it suffices to construct a simplicial functor
$\DK(I) \rightarrow \mM$, where $\mM$ is the simplicial enrichment associated to $\cM$. This can then be replaced by an equivalent strict diagram by \cite{Dwyer-Kan-Smith} (compare \cite[Theorem 4.49]{BVogHI}).
\end{remark}

\begin{defn}\label{def1.8}
Let $X$ be an $\infty$-category. A \emph{strict model} for $X$ is a fibrant simplicial category $\mX$
such that:
\begin{enumerate}
\item{$\Ob(X) = \Ob(\mX)$.}
\item{$\mX(x, y)$ is homotopy equivalent to $\Map_{X}(x, y)$ for each $x, y \in \Ob(X)$.}
\item{Composition in $\mX$ can be identified with composition in $X$. }
\end{enumerate}
\end{defn}

\begin{remark}\label{rmk1.9}
In \cite[\S 4]{Toen}, To{\"{e}}n gave an axiomatization of $\infty$-categories as the fibrant objects in a model category satisfying certain axioms. He then showed (see \cite[Theorem 5.1]{Toen}) that each model category $\cM$ satisfying these axioms admits a Quillen equivalence
$$
q_{M}\colon\cM \rightleftarrows \stSet\colon r_{M}
$$
to the complete Segal model structure of \cite{Rezk}.
It was shown in \cite{BERGNER3} that there is also a zig-zag of Quillen equivalences
$$
\sCat\leftrightarrows \SeCat_{f} \rightleftarrows \stSet
$$
where the middle category denotes the category of Segal precategories equipped with an appropriate model structure. The left adjoint of the first equivalence is given by a functor $F_{0}$ which regards a simplicial category as a bisimplicial set `strictly local' with respect to a certain map $\phi$ (see \cite[\S 8]{BERGNER3}). The left adjoint of the second Quillen equivalence is given by the functor $R\colon\stSet \rightarrow\SeCat_{f}$ which `discretizes' the degree zero part of a simplicial space (see \cite[\S 6]{BERGNER3}).

Thus, the derived functor $E=\bL F_{0} \circ \bR R \circ \bL q_{M} $ provides a reasonable definition of a functorial strict model of an $\infty$-category. Indeed, for the main existing models of $(\infty, 1)$-category theory \ -- \ in particular, for quasi-categories (see \S \ref{sqcic}) \ -- \  the object $EX$ will be a strict model of the $\infty$-category $X$ in the sense of \S \ref{def1.8}.
\end{remark}

\begin{defn}\label{def2.1}
Suppose that $X$ is an $\infty$-category. We call an object $x \in X$ \emph{initial}
if for each $y \in \Ob(X)$, the mapping space $\Map_{X}(x, y)$ is weakly contractible.
We have an evident dual notion of \emph{terminal} object. A \emph{zero object} is
one which is both initial and terminal.
A map of $\Map_{X}(x, y)_{0}$ which factors through a zero object is called a \emph{zero map}.

Thus $f\colon x \rightarrow y$ is a zero map if for each $z \in \Ob(X)$,
both pre-composition $(-) \circ f\colon\Map_{X}(y, z) \rightarrow \Map_{X}(x, z)$ and post-composition $f \circ (-)\colon\Map_{X}(z, x) \rightarrow \Map_{X}( x, y)$ with $f$
are nullhomotopic.

Moreover, any map homotopic (\S \ref{def1.2}) to a zero map is also such, and any two zero maps between $f, g$ are necessarily homotopic.
\end{defn}

We assume from here on that our chosen model of $\infty$-categories has a \emph{join functor } $(-)\ast(-) \colon\cC_{0} \times \cC_{0}$, such that
$\Ob(X\ast Y) = \Ob(Y) \cup \Ob(X)$ and
$$
\Map_{X\ast Y}(x, y) \sim
\left\{
    \begin{array}{ll}
        \Map_{X}(x, y)  & \mbox{if } x,y \in \Ob(X) \\
        \Map_{Y}(x, y)  & \mbox{if } x,y \in \Ob(Y) \\
        \emptyset & \mbox{if } x \in \Ob(Y), y \in \Ob(X) \\
        0 & \mbox{if } x \in \Ob(Y), y \in \Ob(X) \\
    \end{array}
\right.
$$

\begin{defn}\label{def2.2}
Suppose that $F\colon K \rightarrow X$ is a functor of $\infty$-categories. We write $X_{/F}$ for the \emph{slice} $\infty$-category, defined by the universal property that
$\homm(Z, X) =\homm_{F}(K\ast Z, X)$,
where $\hom_{F}(K\ast Z, X)$ is the set of all morphisms such that the composite
$K\ast i \rightarrow K\ast Z \rightarrow X$ is $F$ (here $i$ is the initial object in $\cC_{0}$).

We have an evident dual notion of \emph{coslice} $\infty$-category.
\end{defn}

\begin{defn}\label{def2.3}
A \emph{limit} of a functor $F$ of $\infty$-categories is a terminal object in $X_{/F}$. Dually, the \emph{colimit} of $F$ is initial in the coslice $\infty$-category.
\end{defn}

\begin{defn}\label{def2.4}
A \emph{simplicial object} in an $\infty$-category $X$ is a functor
$ \xd\colon\cB(\Delta\op) \rightarrow X$.  We have a natural tensoring of simplicial sets $K$ over $\xd$, given by the formula:
$$
\xd[K] = \underset{\longleftarrow}{\lim} (\cB(\Delta\op_{/K}) \xrightarrow{p} \cB(\Delta\op) \rightarrow X),
$$
where $\Delta_{/K}\op$ is the simplex category of $K$ (see \cite[I.2]{GJ2}), and $p$ is the projection.
\end{defn}

\begin{defn}
We will call an object $y$ of an $\infty$-category $X$ a \emph{cogroup object} if $\cP(X)(y, x)$ has a group structure, natural in $x \in \Ob(X)$.
\end{defn}

%
%
\sect{The spectral sequence of a simplicial object in an $\infty$-category}
\label{cssso}

As Dwyer, Kan and Stover showed in \cite{Dwyer-Kan-Stover}, the homotopy spectral
sequence of a simplicial space (see \cite[Theorem B.5]{BFrieH}), originally constructed by Quillen in \cite{QuiS} for a bisimplicial group, can be derived from a suitable tower of fibrations, which will allow a formal translation to the $\infty$-category setting.  We start by recalling their construction:

\begin{mysubsection}{The spiral spectral sequence of a simplicial space}
\label{csssec}
Suppose that $D_{\bullet}\colon\Delta\op \rightarrow \sSa$ is a pointed simplicial space.  For each $n\in\NN$, write $Z_{n}(D_{\bullet}):= D_{\bullet}[\Delta^{n}/ \partial \Delta^{n}], C_{n}(D_{\bullet}) := D_{\bullet}[\Delta^{n}/ \Lambda^{n}_{0}]$,
where $D_{\bullet}[K]$ denote the usual tensoring of simplicial spaces over a simplicial set $K$.

We have homotopy fibre sequences
\begin{myeq}\label{eqfibseq}
Z_{n}(D_{\bullet})~\xrightarrow{i}~C_{n}(D_{\bullet})~
  \xrightarrow{\bar{\partial}}~Z_{n-1}(D_{\bullet})~.
\end{myeq}
\end{mysubsection}

\begin{defn}\label{def3.1}
The \emph{spiral spectral sequence} associated to a simplicial space $D_{\bullet}$ is the spectral sequence associated to the exact couple
$$
\xymatrix
{
\oplus_{n \in \NN} \pi_{*}Z_{n}(D_{\bullet}) \ar[dr]_{i_{*}}  && \ar[ll]_{\delta} \oplus_{n \in \NN} \pi_{*}Z_{n}(D_{\bullet}) \\
& \oplus_{n \in\NN} \pi_{*}C_{n}(D_{\bullet}) \ar[ur]_{\bar{\partial}^{*}} &
}
$$
\end{defn}

\begin{theorem}\label{thm3.2} (see \cite[Section 8]{Dwyer-Kan-Stover})
From the $E_{2}$-term onward, the spiral spectral sequence a simplicial space can be identified with the homotopy spectral sequence of Bousfield-Friedlander. 
\end{theorem}

\begin{defn}\label{def2.7}
Suppose that $\mX$ is a fibrant simplicial category. A \emph{homotopical limit} of a map of simplicial categories $f\colon K \rightarrow \mX$ is an object $b$ of $\mX$ together with a collection of maps
$$
\{\eta_{c}\colon b \rightarrow f(c) \}_{c \in \Ob(K)}
$$
such that for each $a \in \Ob(\mX)$ the diagram
$$
\{ \Map_{\mX}(a, b) \rightarrow \Map_{\mX}(a, f(c)) \}_{c \in K}
$$
exhibits $\Map_{\mX}(a, b)$ as a homotopy limit of $\Map_{\mX}(a, -) \circ K$. We have an evident dual notion of \emph{homotopical colimit}. See \cite{DHKSmitM} for further
details.
\end{defn}

\begin{remark}\label{rmk2.8}
For any simplicial category $\mX$, let $\sSet^{\mX}$ denote the category of maps of simplicial categories (where $\sSet$ has its usual simplicial enrichment). As discussed in \cite[A.3.3.13]{Lurie}, the enriched Yoneda embedding identifies homotopical limits in $\mX$ (thought of as an $\infty$-category) with homotopy limits in $\sSet^{\mX}$ equipped with the injective model structure of \cite[A.2.8.2]{Lurie}.

Homotopy limits admit a description in terms of homotopy right Kan extensions, which we can reinterpret as `homotopically terminal' objects in the category of cones over a diagram. Thus, the notion of homotopical limits in a simplicial category, defined above for the Bergner model structure on $\sCat$, fits into our axiomatization of limits in $\infty$-categories in \S \ref{def2.3}.
\end{remark}

\begin{mysubsection}{The spectral sequence of a simplicial object in a simplicial category}
\label{con3.3}
Suppose that $\mX$ is a fibrant simplicial category with countable homotopical limits and colimits, $y$ a fixed homotopy cogroup object of $\mX$, and
$\xd\colon\DK(\Delta\op) \rightarrow\mX$ a pointed homotopy simplicial object
(in the sense of \S \ref{rmk1.7}), with the pointing given by zero maps as in \S  \ref{def2.1}. By \cite[Theorem 4.1]{Dwyer-Kan-Smith}, we can choose an equivalent strict pointed simplicial object $x'_{\bullet}$. We define the spectral sequence associated to $y$ and $\xd$ to be the spiral spectral sequence of \S \ref{def3.1} associated to the simplicial space \ $\pbn\mapsto\Map_{\mX}(y, x'_{n})$.
\end{mysubsection}

\begin{mysubsection}{Simplicial objects in an $\infty$-category}
\label{con2.6}
Let $X$ be an $\infty$-category with countable colimits and limits, satisfying the following two assumptions:
\begin{enumerate}
\item The tensoring takes homotopy pushouts of simplicial sets to pullbacks.
\item The mapping spaces take pullbacks to homotopy pullbacks in the Kan model structure.
\end{enumerate}

We can then mimic the construction of \S \ref{con3.3} for any simplicial object $\xd$ in $X$ and cogroup object $y \in \Ob(X)$, by defining the $n$-th \emph{Moore cycles} object for $\xd$ to be
$Z_{n}(\xd):=\xd[\Delta^{n}/ \partial \Delta^{n}]$,
and the $n$-th \emph{Moore chains} object for $\xd$ to be $C_{n}(\xd):= \xd[\Delta^{n} / \Lambda_{0}^{n} ]$. By our assumptions the cofibration sequence
$$
\Delta^{n-1}/ \partial \Delta^{n-1} \xrightarrow{d^{0}} \Delta^{n} / \Lambda_{0}^{n} \xrightarrow{i} \Delta^{n} / \partial \Delta^{n}
$$
gives rise to a fibre sequence of Kan complexes:
\begin{myeq}\label{eqhfibseq}
\Map_{X}(y, Z_{n}(\xd)) \xrightarrow{i_{n}} \Map_{X}(y, C_{n}(\xd))~
\xrightarrow{\bar{\partial}^{n}}~\Map_{X}(y, Z_{n-1}(\xd))
\end{myeq}
(compare \eqref{eqfibseq}).
Since $y$ is a cogroup object, we obtain from \eqref{eqhfibseq} an exact couple
of groups:
$$
\xymatrix
{
\oplus_{n \in \NN} \pi_{*}\Map_{X}(y, Z_{n}(\emph{x}_{\bullet}))  \ar[dr]_{i_{*}}  && \ar[ll]_{\delta} \oplus_{n \in\NN} \pi_{*}\Map_{X}(y, Z_{n}(\emph{x}_{\bullet})) \\
& \oplus_{n \in\NN} \pi_{*}\Map_{X}(y, C_{n}(\emph{x}_{\bullet})) \ar[ur]_{\bar{\partial}^{*}} &
}
$$
This gives rise to a spectral sequence
$$
E^{1}\sb{n,p} = \pi_{p}\Map_{X}(y, x_{n})~\implies~\pi_{p+n}\Map_{X}(y,\|\xd\|),
$$
which we call the \emph{homotopy spectral sequence} of $\xd$.
We denote the $r$-th differential in this spectral sequence by $\fd^{r}$.
\end{mysubsection}

\begin{remark}\label{rmk3.4}
When working with models of $\infty$-category theory other than simplicial categories, we can often choose a strict simplicial model $\mX$ of an $\infty$-category $X$, and then identify the spectral sequence of \S \ref{con2.6} with the spectral sequence of \S \ref{con3.3} in $\mX$.

In fact, in light of \S \ref{rmk1.9} we could simply \emph{a priori} define the spiral spectral sequence in terms of a strict model for an infinity category, using \S \ref{con3.3}. However, as we will see in the case of quasi-categories in Section \ref{cdssqc}, the framework of \S \ref{con2.6} gives us a way to describe the spectral sequence, and in particular its differentials, internally in any reasonable model of $\infty$-category theory.

Note that in the case the $\infty$-category in question is $\mX = \sSet^{\mathscr{Y}}$ for some simplicial category $\mathscr{Y}$, the two hypotheses of
\S \ref{con2.6} are clearly satisfied, so $\mX(y, -)$ commutes with (homotopy) fibre sequences, and the spectral sequence constructed in \S \ref{con3.3} may be identified with that just described.
\end{remark}

The following result is central for understanding this spectral sequence:

\begin{theorem}\label{thm3.5}
  In the situation of \S \ref{con3.3}, let $\lra{f}\in E^{1}\sb{n,p} = \pi_{0}\Map_{\mX}(\Sigma^{p}y,x_{n})$ be represented by $f\colon\Sigma^{p}y \rightarrow x_{n}$. Then $\lra{f}$ survives to $E^{r}\sb{n,p}$ for $r\geq 2$ if and only if we can find a homotopy coherent diagram $\DK(\pbDel\op_{n-r+1, n} \times\bbo) \rightarrow \mX$ as in \eqref{diagnk}.
\end{theorem}

This follows directly from Theorem \ref{thm4.7} below, which provides a description of the
differentials in the spectral sequence, our other main goal.

\begin{remark}\label{altdia}
A homotopy coherent diagram as in \eqref{diagnk} is the same as a homotopy coherent
truncated restricted simplicial object lifting the diagram
\myzdiag[\label{diagnkeq}]{
  \Sigma^{p}y \ar@<3ex>[rr]^{d_{0} = f} \ar@<0.5ex>[rr]^{d_{1} = 0} \ar@<-2.5ex>[rr]_{d_{n+1} = 0}^{\vdots} && x_{n} \ar@<2ex>[rr]^{d_{0}} \ar@<-1ex>[rr]_{d_{n}}^{\vdots} && x_{n-1}
  \ar@<2ex>[rr]^{d_{0}} \ar@<-1ex>[rr]_{d_{n-1}}^{\vdots}&& \cdots
  \ar@<2ex>[rr]^{d_{0}} \ar@<-1ex>[rr]_{d_{n-r+1}}^{\vdots} && x_{n-r}
}
in $\cP(\mX)$.
\end{remark}

\begin{mysubsection}{The filtration in the spectral sequence}
\label{sfiltss}
We can regard the homotopical colimit $\|\xd\|$ as providing an augmentation of the simplicial object $\xd$. Thus, Theorem \ref{thm3.5} also applies to the case $r = n+1$, yielding a necessary and sufficient condition for an element which survives to the $E^{\infty}$-page of the spectral sequence to be represented in a given filtration. More precisely, Theorem \ref{thm4.7} provides a series of obstructions to an element at the $E^{1}$-page to appear in the appropriate filtration degree.

Specializing to the category of pointed spaces, the last of these obstructions allows one to study how homotopy groups of spaces \ -- \ in particular, spheres \ -- \ can be generated under certain `higher homotopy operations' from certain indecomposables (see \cite{BBSenHS}).
\end{mysubsection}

%
%
\sect{Differentials in the spectral sequence of a simplicial space}
\label{cdvho}

In order to study the differentials in the spectral sequence of Section \ref{cssso},
we first characterize them in the case of a strict simplicial object.
We start by recalling the following result (dual to \cite[X, Prop.~6.3]{BK}):

\begin{lemma}[\protect{\cite[Lemma 2.7]{StoV}}]\label{lem4.1}
If $\xd$ is a Reedy fibrant simplicial space, the natural map $\pi_{\ast}C_{n}(\xd) \rightarrow C_{n}\pi_{\ast}(\xd)$ is an isomorphism.
\end{lemma}

Here, $C_{n}(\xd)$ is as in \S \ref{csssec}, and $C_{n}\pi_{\ast}(\xd)$ denotes the intersection of the kernels of $d_{i}\colon\pi_{*}(x_{n})\rightarrow \pi_{*}(x_{n-1})$ for $1\leq i\leq n$.

\begin{mysubsection}{Differentials for Reedy fibrant simplicial objects}
\label{con4.2}
Suppose that we are in the situation of \S \ref{con3.3} and that $\Map_{\mX}(y, \xd)$
is a Reedy fibrant (strict) simplicial space. Any element $\lra{f}\in E^{2}\sb{n,p}$ is represented by $f \in \Map_{\mX}(\Sigma^{p}y, C_{n}(\xd))_{0}$. Using the notations:
$$
C_{n}(x_{\bullet}) \xrightarrow{\bar{\partial}^{n}} Z_{n-1}(x_{\bullet}) \xrightarrow{i_{n-1}} C_{n-1}(x_{\bullet}) \xrightarrow{j_{n-1}} x_{n-1}
$$
for the obvious maps, we set $g:=\bar{\partial}^{n} \circ f$.

By Lemma \ref{lem4.1}, any nullhomotopy $H'\colon j_{n-1} \circ
i_{n-1} \circ g\sim 0$ lifts to a nullhomotopy $H\colon i_{n-1}
\circ g  = i_{n-1} \circ \bar{\partial^{n}} \circ f\sim 0$. Thus a
necessary and sufficient condition for $\lra{f}$ to survive to the
$E^{2}$-page is the existence of $H'$ as above.

Since $\bar{\partial}^{n-1} \circ \bar{\partial}^{n} = 0$, when $H'$ (and thus $H$) exists,
$\bar{\partial}^{n-1} \circ H$ is a loop on $\Map_{\mX}(\Sigma^{p}y, C_{n}(\xd))$, and $\fd^{2}\lra{a}$ is represented by
$i_{n-2}\circ\bar{\partial}^{n-1} \circ H\colon\Sigma^{p+1}y \rightarrow C_{n-2}(\xd)$.
We then have a commutative diagram
$$
\xymatrix
{
C_{n-1}(\xd) \ar[r]_{\bar{\partial}^{n-1}} \ar[d]_{j_{n-1}} & Z_{n-2}(\xd) \ar[d]^{j_{n-2} \circ i_{n-2}} \\
x_{n-1} \ar[r]_{d_{0}} & x_{n-2},
}
$$
and
$$
j_{n-2} \circ i_{n-2} \circ \bar{\partial}^{n-1} \circ H = d_{0} \circ j_{n-1} \circ H = d_{0} \circ H'.
$$
This allows us to identify $\fd^{2}\lra{a}$ with $d_{0}H'$.

If $\fd^{2}\lra{a}$ vanishes, we have a nullhomotopy $H''\colon d_{0}(H') \sim 0$, and
$\lra{f}$ survives to the $E^{3}$-page of the spectral sequence, with $\fd^{3}\lra{f}$ identified with $d_{0} \circ H''$. Repeating this process yields $\fd^{r}\lra{f}$ for all $r$
(when defined): in each case we have simply choose liftings in the
long exact sequence of a fibration.
This completes our description of the differentials in the Reedy fibrant case.
\end{mysubsection}

\begin{mysubsection}{Differentials for an arbitrary simplicial space}
\label{exam4.3}
Suppose now that we have an arbitrary strict simplicial space of the form
$\pbn\mapsto \Map_{\mX}(y, x_{n})$, where $\mX$ is a fibrant simplicial category.
We can then represent each element $\lra{f} \in E^{1}\sb{n,p}$ in the spectral sequence of
\S \ref{con3.3} by a map $f\colon\Sigma^{p}(y) \rightarrow x_{n}$.

We define the \emph{$n$-th matching object} for $\xd$ to be
$M_{n}(\xd):=\xd[\partial \Delta^{n}]$, and its \emph{$n$-th modified matching object})
to be $\hat{M}_{n}(\xd):= \xd[\Lambda_{0}^{n}]$. We then have a homotopy fiber sequence
\begin{myeq}\label{fibre1}
  C_{n-1}(\xd) \rightarrow x_{n-1} \xrightarrow{\delta_{n-1}} \hat{M}_{n-1}(\xd)
\end{myeq}
in $\mX$, and hence homotopies $G_{j}\colon d_{i} \circ f \sim 0$ for $j \ge 1$, as well as \emph{coherence homotopies} $F_{i, j}\colon d_{i}G_{j} \sim d_{j-1}G_{i}$
for all $1 \le i < j \le n$.

If $\fd^{1}\lra{f} = 0$, we also have a nullhomotopy $H\colon d_{0} \circ f \sim 0$,
and the differential $\fd^{2}\lra{f}$ can be represented by the map $\partial \Delta^{2} \rightarrow \Map_{\mX}(\Sigma^{p}y, x_{n-2})$  depicted by:
\mytdiag[\label{dee2strict}]{
& \ar[ddl]_{d_{0}G_{1}} d_{0}d_{1}f \ar[ddr]^{d_{0}H} & \\
& & \\
 \ar[rr]_{\Id} 0 & & 0.
}

To see this, by Reedy fibrant replacement of $\xd$ we can reduce to the case where \eqref{fibre1} is a strict limit and $G_{1}$ is the identity nullhomotopy. We then
recover the description of \S \ref{con4.2}.

Assuming now $\fd^{2}\lra{f}$ vanishes, we can find a fill-in $K\colon\Delta^{2} \rightarrow \Map_{\mX}(\Sigma^{p}y, x_{n-2})$ for \eqref{dee2strict}, and the vanishing of the
differential $\fd^{2}$ implies that we have additional coherences
$F_{0, j}\colon d_{0}G_{j} \sim d_{j-1}G_{0}$
for $G_{0} = H$, since $d_{0}f$ factors through $Z_{n-1}(\xd)$, and in the Reedy fibrant case we have a fibre sequence
$$
Z_{n-1}(\xd) \rightarrow x_{n-1} \rightarrow M_{n-1}(\xd)~.
$$

Now the value of the differential $\fd^{3}\lra{f}$ is given by the following $3$-simplex boundary in $\Map_{\mX}(\Sigma^{p}y, x_{n-3})$:
\mydiagram[\label{dee3strict}]{
& & 0 \ar@{-}[dd]_{\rotatebox{90}{$\scriptstyle d_{0}d_{1}G_{0}$}} & &  \\
\ar@{}[drrr]_{d_{0}F_{0, 2}} \ar@{}[ddrrrr]_>>>>>>>>>>>>>>>>>>>>{d_{0}F_{1, 2}} & &  & & \ar@{}[dlll]^{d_{0}K} \\
& & 0 \ar@{-}[drr] \ar@{-}[dll] & & \\
0 \ar@{-}[uuurr]^{\Id_{0}} \ar@{-}[rrrr]_{\Id_{0}}  \ar@{-}[urr]  &  &  \ar@{}[ul]^>>>>>>{\rotatebox{30}{$\scriptstyle d_{0}d_{1}G_{2}$}}  \ar@{}[ur]_>>>>>>{\rotatebox{-30}{$\scriptstyle d_{0}d_{1}G_{1}$}}  & & \ar@{-}[uuull]_{\Id_{0}} 0,
}
with the bottom face mapping to $0$. Again, in the Reedy fibrant case the higher coherences $F_{1, 2}$ and $F_{0, 2}$ vanish, and we recover the description in \S \ref{con4.2}.

This procedure can be continued to obtain all higher differentials in this case, though we shall not require an explicit formula for them.
\end{mysubsection}

%
%
\sect{The Dwyer-Kan resolution of $\Delta\op$}
\label{cdkrd}

By Definition \ref{def2.4}, a simplicial object in an $\infty$-category $X$ is
a functor $\xd\colon\cB(\Delta\op) \rightarrow X$. However, for the purpose of constructing the homotopy spectral sequence for $\xd$, we need
only consider the induced homotopy coherent simplicial space $\Map_{X}(y,\xd)$.
To avoid obscuring the salient points, we restrict attention for the
rest of this section to the case where $\xd$ is a homotopy-coherent simplicial
object in a fibrant simplicial category $\mX$: that is, a simplicial functor
$\DK(\Delta\op)\to\mX$.

From the description of the differentials in Section \ref{cdvho} we see that the degeneracy maps are not needed \ -- \ so in fact in order to compute for any specific
differential we need only a finite subcategory
of $\Delta$ of the form $\bDel_{k,n}$. In order to cover also the filtration degree (see \S \ref{sfiltss}), we state our results in terms of the categories $\pbDel_{n}$ (see \S \ref{snac}).

The simplicial mapping spaces in $\DK(\pbDel_{k,n})$ can be understood better
by observing that they are triangulations of a certain well-known sequence of polytopes:

\begin{mysubsection}{Permutahedra}
\label{sconcpoly}
Recall that an \emph{$n$-permutahedron} $\mP_{S}$ is a convex
polytope spanned by the $(n+1)!$ points in $\RR^{n+1}$ obtained by
permuting the coordinates of $(x_{0},x_{1},\cdots,x_{n})$ (where
$S=\{x_{0},x_{1},\cdots,x_{n}\}$ are $n+1$ distinct real numbers).
Thus the $1$-permutahedron is an interval, the $2$-permutahedron is a hexagon, and the
$3$-permutahedron is depicted in Figure \ref{ftperm}.

\stepcounter{theorem}
\begin{figure}[htbp]
\begin{picture}(200,195)(50,-10)
%
%
\put(40,72){\circle*{5}}
\put(0,69){\tiny $d_0d_1d_2d_2$}
\put(41,69){\line(1,-3){10}}
\put(52,36){\circle*{5}}
\put(13,33){\tiny $d_0d_1d_1d_2$}
\put(54,36){\line(1,0){33}}
\put(88,36){\circle*{5}}
\put(57,42){\tiny $d_0d_1d_1d_1$}
\put(89,37){\line(2,3){22}}
\put(112,72){\circle*{5}}
\put(74,70){\tiny $d_0d_1d_2d_1$}
\put(41,74){\line(1,3){10}}
\put(52,108){\circle*{5}}
\put(13,105){\tiny $d_0d_1d_2d_3$}
\put(55,108){\line(1,0){45}}
\put(100,108){\circle*{5}}
\put(62,111){\tiny $d_0d_1d_1d_3$}
\put(101,105){\line(1,-3){10}}
%
%
\put(54,34){\line(4,-3){47}}
\put(100,0){\circle*{5}}
\put(58,-4){\tiny $d_0d_0d_1d_2$}
\put(103,0){\line(1,0){21}}
\put(124,0){\circle*{5}}
\put(128,-8){\tiny $d_0d_0d_1d_1$}
\put(91,32){\line(1,-1){33}}
%
%
\put(128,0){\line(4,1){45}}
\put(172,12){\circle*{5}}
\put(177,6){\tiny $d_0d_0d_0d_1$}
\put(173,15){\line(1,4){8}}
\put(184,60){\circle*{5}}
\put(170,50){\tiny $d_0d_1d_0d_1$}
\put(115,73){\line(3,1){34}}
\put(148,84){\circle*{5}}
\put(127,71){\tiny $d_0d_0d_2d_1$}
\put(151,82){\line(3,-2){32}}
%
%
\put(174,14){\line(1,1){45}}
\put(220,60){\circle*{5}}
\put(225,57){\tiny $d_0d_0d_0d_0$}
\put(221,64){\line(1,4){10}}
\put(186,62){\line(1,1){10}}
\put(205,81){\line(1,1){23}}
\put(232,108){\circle*{5}}
\put(235,105){\tiny $d_0d_1d_0d_0$}
%
%
\put(136,120){\circle*{5}}
\put(100,124){\tiny $d_0d_0d_1d_3$}
\put(137,117){\line(1,-3){6}}
\put(144,96){\line(1,-3){2}}
\put(138,122){\line(3,4){33}}
\put(172,168){\circle*{5}}
\put(177,171){\tiny $d_0d_0d_0d_3$}
\put(175,167){\line(2,-1){44}}
\put(220,144){\circle*{5}}
\put(225,141){\tiny $d_0d_0d_2d_0$}
\put(221,141){\line(1,-3){10}}
%
%
\put(103,109){\line(3,1){32}}
\put(54,111){\line(2,3){21}}
\put(76,144){\circle*{5}}
\put(37,141){\tiny $d_0d_0d_2d_3$}
\put(79,145){\line(2,1){45}}
\put(124,168){\circle*{5}}
\put(112,173){\tiny $d_0d_1d_0d_3$}
\put(127,168){\line(1,0){45}}
%
%
\multiput(77,140)(1,-4){12}{\circle*{.5}}
\put(88,96){\circle*{3}}
\put(60,87){\tiny $d_0d_0d_2d_2$}
\multiput(91,95)(3,-1){12}{\circle*{.5}}
\put(124,84){\circle*{3}}
\put(112,90){\tiny $d_0d_1d_0d_2$}
\multiput(127,86)(3,2){2}{\circle*{.5}}
\multiput(145,98)(3,2){6}{\circle*{.5}}
\put(160,108){\circle*{3}}
\put(165,105){\tiny $d_0d_1d_1d_0$}
\multiput(161,111)(1,3){12}{\circle*{.5}}
\put(172,144){\circle*{3}}
\put(171,134){\tiny $d_0d_1d_2d_0$}
\multiput(168,146)(-4,2){12}{\circle*{.5}}
%
%
\multiput(44,74)(4,2){6}{\circle*{.5}}
\put(84,94){\circle*{.5}}
%
%
\multiput(123,81)(1,-3){12}{\circle*{.5}}
\put(136,48){\circle*{3}}
\put(135,40){\tiny $d_0d_0d_0d_2$}
\multiput(133,46)(-3,-4){12}{\circle*{.5}}
%
%
\multiput(141,50)(3,2){12}{\circle*{.5}}
\put(172,72){\circle*{3}}
\put(174,74){\tiny $d_0d_0d_1d_0$}
\multiput(175,71)(4,-1){12}{\circle*{.5}}
%
%
\multiput(162,105)(1,-3){12}{\circle*{.5}}
\multiput(176,144)(4,0){12}{\circle*{.5}}
\end{picture}
\caption[ftperm]{The $3$-permutahedron}
\label{ftperm}
\end{figure}
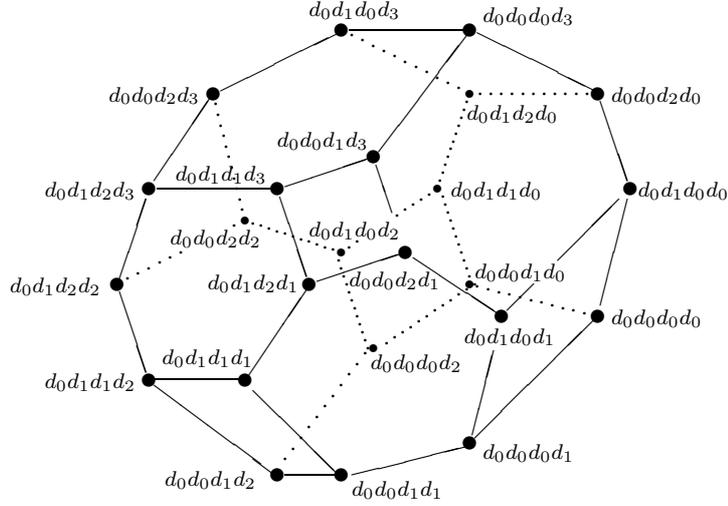

A \emph{face} of $\mP_{S}$ is a convex polytope obtained by intersecting the boundary
$\partial\mP_{S}$ with half-spaces in $\RR^{n+1}$. In particular, $\mP_{S}$ has one \emph{facet} (that is, $(n-1)$-dimensional face), isomorphic to the product $\mP_{S_{1}} \times \mP_{S_{2}}$, for each ordered partition  $S_{1} \amalg S_{2}$ of $S=\{0, \cdots, n \}$, say. More generally, the $k$-dimensional faces are in one-to-one correspondence with ordered partitions of $S$ into $n-k+1$ disjoint sets.

As we shall see in Proposition \ref{lema.3} below, for each $n\geq 1$ and $k\geq -1$, every component of the simplicial mapping space $\DK(\pbDel\op)(\ppp{\bk+\bn+\bo},\pbk)$
of the Dwyer-Kan resolution constitutes a canonical triangulation of the $n$-permutahedron
$\mP_{S}$, which we denote by $\mP^{n}$.
\end{mysubsection}

\begin{examples}\label{exama.1}

\begin{enumerate}
\item Evidently $\DK(\pbDel\op)(\pbm, \ppp{\bm-\bo})$ is discrete, with vertices corresponding to the morphisms $d_{i} \colon \pbm\rightarrow\ppp{\bm-\bo}$.
\item Any map $\theta\colon\pbm\rightarrow\ppp{\bm-\btw}$ in $\pbDel\op$ can be written uniquely as $d_{i}d_{j}$ for some $i < j$. The corresponding component of $\DK(\pbDel \op)(\pbm, \ppp{\bm-\btw})$ is the subdivided interval:
$$
\xymatrix
{
(d_{i})(d_{j}) \ar@{-}[r] & (d_{i}d_{j}) = (d_{j-1}d_{i}) & \ar@{-}[l] (d_{j-1})(d_{i})~.
}
$$
\item Any map $\theta\colon\pbm\rightarrow\ppp{\bm-\bth}$ has the form $d_{i}d_{j}d_{k}$
  for $i < j < k$, and the corresponding component of $\DK(\pbDel \op)(\pbm,\ppp{\bm-\bth})$ is the subdivided hexagon:
\mytdiag[\label{eqsubhexagon}]{
& & (d_{k-2})(d_{j-1})(d_{i}) \ar@{-}[ddrr] \ar@{-}[ddll] & & \\
 \ar@{}[urr]^<<<{ (d_{k-2})(d_{i}d_{j}) = (d_{k-2})(d_{j-1}d_{i}) \, \, \, \, }&  & &  & \ar@{}[ull]_<<<{\, \, \,  (d_{j-1}d_{k-1})(d_{i}) = (d_{k-2}d_{j-1}) (d_{i}) } \\
(d_{k-2})(d_{i})(d_{j}) \ar@{-}[dddd] & & &  & (d_{j-1})(d_{k-1})(d_{i}) \ar@{-}[dddd]   \\
& & & & \\
\ar@{}[r]_<{ (d_{i}d_{k-1})(d_{j}) = (d_{k-2}d_{i})(d_{j}) \, \, \, \, } & & (d_{i}d_{j}d_{k}) \ar@{.}[uurr] \ar@{.}[ddrr] \ar@{.}[uull] \ar@{.}[dddd]  \ar@{.}[uuuu] \ar@{.}[ddll] \ar@{.}[rr]  \ar@{.}[ll] \ar@{.}[uuur] \ar@{.}[uuul]  \ar@{.}[dddr] \ar@{.}[dddl] & & \ar@{}[l]^<{\, \, \,  (d_{j-1})(d_{k-1}d_{i}) = (d_{j-1})(d_{i} d_{k}) } \\
& & & & \\
(d_{i})(d_{k-1})(d_{j}) \ar@{-}[ddrr] & & & & (d_{j-1})(d_{i})(d_{k}) \ar@{-}[ddll] \\
& \ar@{}[r]_<{ (d_{i})(d_{j}d_{k}) = (d_{i})(d_{k-1} d_{j}) \, \, \, \, } & &  \ar@{}[l]^<{\, \, \,  (d_{j-1}d_{i})(d_{k}) = (d_{i}d_{j})( d_{k} )}& \\
& & (d_{i})(d_{j})(d_{k}) & &
}
The original vertices correspond to all complete decompositions of $\theta$ (as in Figure \ref{ftperm}), while the new vertices (the midpoints of the edges and the center)
correspond to partial decompositions.
\end{enumerate}
\end{examples}

\begin{lemma}\label{lema.2}
The set of factorizations of a morphism $\theta\colon\pbn\rightarrow\pbm$ as $\pbn\xrightarrow{\theta'}\ppp{\bm'}\xrightarrow{\theta''}\pbm$  in $\pbDel\op$ are in one-to-one correspondence with the set of ordered partitions of the set $\{ 0, \cdots , n-m-1 \}$.
\end{lemma}

\begin{proof}
Note that by \S \ref{snac}, maps $[\bn] \rightarrow [\bm]$ in $\pbDel\op$ correspond to injections of ordered sets $[\bm] \rightarrow [\bn]$.

First suppose that we have a factorization $\theta'' \theta'$ as above. Note that there is an isomorphism of ordered sets $Q\colon\pbn-\Image(\theta) \cong\ppp{\bn-\bm-\bo}$. If $C$ is the complement of $\Image(\theta)$ in $\Image(\theta')$, then $Q(C)$ gives us an ordered subset $S_{1}$ of $\ppp{\bn-\bm-\bo}$.

On the other hand, suppose we are given an ordered subset $S_{1}$ of $\ppp{\bn-\bm-\bo}$. Then there is a unique factorization of $\theta$ as a pair of injections of ordered sets
$$
\pbm \rightarrow \ppp{\bm+ |S_{1}|}  \rightarrow \pbn~,
$$
so that the image of the second map is $Q^{-1}(S_{1}) \cup \Image(\theta)$.

These constructions are evidently inverse to each other, and ordered partitions
of $\{ 0, \cdots, n-m-1 \}$ are in one-to-one correspondence with choice of ordered subsets of $\{ 0, \cdots n-m-1\}$.
\end{proof}

\begin{prop}\label{lema.3}
The simplicial set $\DK(\pbDel \op)(\pbl,\pbm)$ is isomorphic to a disjoint union
$\coprod_{\theta\colon\pbl \rightarrow\pbm} \mP^{j-m-1}$ of triangulations of the $(j-m-1)$-permutahedron.
\end{prop}

\begin{proof}
The proof is by induction on $j-m$. The cases of $j-m \le 3$ are covered by Examples \ref{exama.1}.

Suppose that the statement holds for all $j-m$ up to $k$, and assume
$j-m = k+1$. Consider the component $C_{\theta}$ of $\DK(\pbDel \op)(\pbl,\pbm)$ corresponding to $\theta\colon \ppp{\bl} \rightarrow \ppp{\bm}$.  The boundary of $C_{\theta}$ is given by the following coequalizer
\mydiagram[\label{DKmapping1}]{
\coprod_{\theta_{1}\theta_{2}\theta_{3} = \theta} C_{\theta_{1}} \times C_{\theta_{2}} \times C_{\theta_{3}} \ar@<1ex>[rr]^{a_{1}} \ar@<-1ex>[rr]_{a_{2}} && \coprod_{\gamma_{1}\gamma_{2}} C_{\gamma_{1}} \times C_{\gamma_{2}} \ar[rr]_{b} && \partial C_{\theta}~,
}
where the maps $a_{1}$, $a_{2}$, and $b$ are induced by composition. By Lemma \ref{lema.2},
the factorizations of $\theta$ in $\pbDel \op$ are in one-to-one correspondence with partitions of $\{ 0, \cdots j-m-1 \}$. Combining this with the inductive hypothesis, \eqref{DKmapping1} can be identified with
\myvdiag[\label{DKmapping2}]{
\displaystyle \coprod_{P_{1} \cup P_{2} \cup P_{3} = [l-m-1]} \mP^{|P_{1}|} \times  \mP^{|P_{2}|}  \times  \mP^{|P_{3}|} \ar@<1ex>[r] \ar@<-1ex>[r] & \displaystyle \coprod_{Q_{1} \cup Q_{2} = [j-m-1]}  \mP^{|Q_{1}|} \times  \mP^{|Q_{2}|} \ar[r] &\partial C_{\theta}~,
}
where the two left maps are now induced by face inclusion. This coequalizer is $\partial \mP^{j-m-1}$. But $C_{\theta}$ is isomorphic to (a triangulation of) the cone over its boundary (i.e. with cone point corresponding to $\theta$ enclosed in $j-m$ sets of brackets).
Because it is convex, the permutahedron is also the cone over the union of its facets; the result follows.
\end{proof}

%
%
\sect{Differentials for a coherent simplicial object}
\label{cdsscso}

We now describe the differentials for a simplicial object $\xd$ in
a pointed $\infty$-category $X$. This essentially reduces to the case where $X$ is $\Kan$ (see \S \ref{snac}), with the usual (fibrant) enrichment, because
$\Map_{X}(\Sigma^{p}y, \xd)$ is in fact a coherent simplicial space.

For any $u,v\in\Ob(X)$, the distinguished choice of zero map in each pointed space $\Map_{X}(u,v)$ is denoted by `0'.

\begin{lemma}\label{lema.4}
Suppose we have a diagram of simplicial sets
$$
\xymatrix
{
A\ar@{>->}[d]_{i} \ar[r] & X\\
B \ar@{.>}[ur] &
}
$$
where $X$ is a Kan complex. If we can find a lift
$$
\xymatrix
{
|A| \ar[r] \ar[d]_{|i|} & |X| \\
|B| \ar[ur] &
}
$$
in the geometric realization of the diagram,
then the lift exists in the original diagram.
\end{lemma}

\begin{proof}
Let $\Sing$ denote the right adjoint of geometric realization. We have a diagram
$$
\xymatrix
{
A \ar[d] \ar[dr] \ar[rr] && X \ar[dr]_{a}  & \\
B \ar@{.>}[urr] \ar[dr]_{s}& \Sing|A| \ar[rr] \ar[d]^{\Sing|i|} && \Sing|X| \\
& \Sing|B| \ar[urr] &&
}
$$
 $X \rightarrow \Sing|X|$ is a weak equivalence of Kan complexes by \cite[Theorem I.10.10 and Proposition I.11.3]{GJ2}. Thus, we have a commutative diagram
 $$
 \xymatrix
 {
 A \ar[d]_{[i]} \ar[rr] && X \\
 B \ar[urr]_{[a]^{-1} [\Sing|i|] [s]}&&
}
 $$
 in the homotopy category of simplicial sets. The result now follows from \cite[Lemma A.2.3]{Lurie}.
\end{proof}

\begin{corollary}\label{cora.5}
Suppose that we have a map of pointed simplicial sets $\phi\colon\partial \mP^{n} \rightarrow X$. Then the element of $\pi_{n-1}(X)$ determined by $\phi$ vanishes if and only if we can find a lift in the diagram
$$
\xymatrix
{
\partial \mP^{n} \ar[d] \ar[r] & X \\
\mP^{n} \ar@{.>}[ur] &
}
$$
\end{corollary}

\begin{mysubsection}{The maps $\mT_{\theta}$}
\label{con4.5}
We are now in a position to describe the central construction in our description of the differentials:

Suppose that we have a diagram $D\colon\DK(\pbDel\op_{n-r+1, n} \times\bbo) \rightarrow \mX$
representing $\lra{f}\in E\sp{1}\sb{n,p}$ as in \eqref{diagnk} (or the equivalent description \eqref{diagnkeq}).
For each map $\theta\colon\ppp{\bnp}\rightarrow\ppp{\bn-\br}$ in $\pbDel\op$, we now construct a map of pointed spaces $\mT_{\theta}\colon\partial\mP^{r}\rightarrow \Map_{\mX}(\Sigma^{p}y, x_{n-r})$. If $\mT_{\theta}$ extends to the interior of $\mP^{r}$, we choose an extension
and denote it by $\bmT_{\theta}$.

The map $\mT_{d_{0}\cdots d_{r}}$  determines an element of $\pi_{r}\Map_{\mX}(\Sigma^{p}y, x_{n-r})$, by the loop-suspension adjunction. We will show in Theorem \ref{thm4.7} that this is represents $\fd^{r}\lra{f}$.

As in Section \ref{cdkrd}, the vertices of the permutahedron correspond to factorizations of $\theta$ as face maps. For $r > 1$, the permutahedron will be pointed by the unique vertex of the form $d_{i_{0}} d_{i_{1}} \cdots d_{i_{r}}$ with $i_{0} < i_{1} \cdots < i_{r}$.

We begin our inductive procedure by setting $\bmT_{d_{0}} = f$, and $\bmT_{d_{i}} = 0$ for
$i\geq 1$. For each $i < j$ we have
$$
\bmT_{d_{i}d_{j}} =
\begin{cases}
        \Id_{0}  & \mbox{if } i > 0 \\
        G_{j} & \mbox{if } i = 0
    \end{cases}
$$
In either case, we think of $\bar{\mT}_{d_{i}d_{j}}$ as a map pointed by the zero map $d_{i}d_{j} = 0$.

The first non-trivial case of this construction is $\theta\colon\ppp{\bnp}\rightarrow\ppp{\bn-\btw}$.
In this case, we write $\theta = d_{i}d_{j}d_{k}, i < j < k$. If $i\geq 1$, we set $\mT_{\theta} = 0$, so it trivially extends to $\bar{\mT}_{\theta}$. Otherwise, we let $\mT_{d_{0}d_{j}d_{k}}$ denote the map from the boundary of the hexagon $\partial \mP^{2}$ constructed in \eqref{dee2lax}. Thus, $\mT_{d_{0}d_{1}d_{2}}$ represents $\fd^{2}\lra{f}$, and $\bar{\mT}_{d_{0}d_{1}d_{2}}$ exists if and only if $\fd^{2}\lra{f}$ vanishes by Corollary \ref{cora.5}.

In general, suppose that $\bmT_{\theta''}$ exists for each $\theta''\colon\ppp{\bnp} \rightarrow\ppp{\bn-\br+\bo}$.
For each $\theta'\colon\pbm \rightarrow \ppp{\bm'}$ with $m \le n$, we write $H_{\theta'}$ for the image of the component corresponding to $\theta'$ of $\DK(\pbDel\op)(m, m')$ under
$$
\xd\colon\DK(\pbDel\op)(m, m') \rightarrow \Map_{\mX}(x_{m}, x_{m'}).
$$

Let $\theta\colon\ppp{\bnp} \rightarrow\ppp{\bn-\br}$. The facets of $\mP^{r}$ correspond to factorizations $(\theta') (\theta'')$ by Proposition \ref{lema.3} and its proof. We let $\mT_{\theta}$ be the map $\partial\mP^{r} \rightarrow \Map_{X}(\Sigma^{p}y, x_{n-r})$, whose value on the facet $(\theta')(\theta'')$ of $\mP^{r}$ is the composite of
$$
\bar{\mT}_{\theta''} \times H_{\theta'} \subseteq \Map_{\mX}(\Sigma^{p}y, x_{m}) \times \Map_{\mX}(x_{m}, x_{n-r}) \rightarrow \Map_{\mX}(\Sigma^{p}y, x_{n-r})
$$
(inclusion followed by the composition map).
\end{mysubsection}

\begin{example}\label{exam4.4}
  Let $\xd$ be a homotopy coherent simplicial object in the simplicial category $\mX$, and suppose that we have an element $\lra{f}\in\pi_{0}\Map_{\mX}(\Sigma^{p}y,x_{n})$
represented by $f\colon\Sigma^{p}(y) \rightarrow x_{n}$.
  If $\fd^{1}\lra{f}=0$, $\lra{f}$ represents an element in $E^{2}\sb{n,p}$ of the spiral spectral sequence associated to $y$ and $\xd$, we have a homotopy commutative diagram
as in \eqref{altdia} with $r = 2$.

Since $\xd$ is homotopy coherent, we have homotopies
$H_{i, j} \colon d_{i}d_{j} \simeq d_{j-1}d_{i}\colon x_{n} \rightarrow x_{n-2}$ for $i < j \le  k$.
These yield a map $K_{i,j} \colon\partial \mP^{2} \rightarrow \Map_{X}(\Sigma^{p}y, x_{n-2})$,
on boundary of the hexagon, described by the following diagram:
\mysdiag[\label{dee2lax}]{
&  \text{\footnotesize$d_{k-2}d_{j-1}f$}   & \\
0 = d_{k-2}d_{0}d_{j} \ar@{-}[ur]^{d_{k-2}G_{j}}& & \ar@{-}[ul]_{H_{j-1, k-1}f} d_{j-1}d_{k-1}f\\
0 = d_{0}d_{k-1}d_{j} \ar@{-}[u] & & \ar@{-}[u]_{d_{0}G_{k}} d_{j-1}d_{0}d_{k} = 0 \\
& \ar@{-}[ur] 0 = d_{0}d_{j}d_{k} \ar@{-}[ul] &
}

In the strict simplicial case, $H_{0, 1}$ is just the identity, so $K_{0, 1}$
reduces to \eqref{dee2strict}. Thus, it follows $K_{0, 1}$ indeed represents the value of the differential $\fd^{2}$ applied to $\lra{f}$ in the spiral spectral sequence.

If $\fd^{2}\lra{f}$ vanishes, we have an extension
$\bar{K}_{0, 1} \colon\mP^{2} \rightarrow \Map_{\mX}(\Sigma^{p}y, x_{n-2})$ to the interior of
the hexagon \eqref{dee2lax}.
In this case, we claim that $\fd^{3}\lra{f}$ is represented by a map $\partial \mP^{3} \rightarrow \Map_{\mX}(\Sigma^{p}y, x_{n-3})$ as depicted in Figure \ref{fopenperm}.
This is a planar version of the boundary of the $3$-permutahedron of Figure \ref{ftperm}, in which we have replaced the first (rightmost) $d_0$ by $f$, and all other rightmost face maps by $0$, using the equivalence of \eqref{diagnk} to \eqref{altdia} in order to think of the former as a restricted simplicial object in $\mX$. One of the squares of Figure \ref{ftperm}, appears as the exterior of Figure \ref{fopenperm} (indexed by $0$).
The maps from the faces of the permutahedron with the specified
boundaries automatically exist, except possibly for those labelled
$(d_{0})(d_{0}d_{1}f)$ and $(d_{0})(d_{0}d_{2}f)$.

The maps  $H_{i j}$ and $H_{i j k}$ are higher coherences (coming from the fact that $\xd$ is $\infty$-homotopy coherent, in the terminology of \cite{Dwyer-Kan-Smith});
in the strict simplicial case, $H_{i, j}$ is the identity homotopy, and the maps so labelled can be identified with the faces labelled $F_{0, 2}$ and $F_{1, 2}$ in \eqref{dee3strict} above. Thus the composite homotopies $H_{i,j}G_{k}$ reduce simply to $G_{k}$, and the diagram above reduces to \eqref{dee3strict}.
This indicates that the differentials described in \S \ref{exam4.3} for strict simplicial objects are a special case of those described here in the more general (coherent) case.
\end{example}

\stepcounter{theorem}
\begin{figure}[htbp]
\begin{center}
\begin{tikzpicture}
                 \node  at (0, 0) (A) {\text{\tiny$0$}};
         \node at (0, 2) (B) {\text { \tiny$d_{0}d_{0}d_{0}f$}};
         \node at (2, 0) (C) {\text{\tiny$0$}};
         \node at (2, 2) (D) {\text { \tiny$d_{0}d_{1}d_{0}f$}};
         \node at (3, 3.628) (E) {\text{\tiny$d_{0}d_{0}d_{2}f$}};
         \node at (4, 2) (F) {\text{\tiny$0$}};
         \node at (3, -1.628) (G) {\text{\tiny$0$}} ;
         \node at (4, 0) (H) {\text{\tiny$0$}};
         \node at (-1, 3.628) (2HEX1) {\text { \tiny$d_{0}d_{0}d_{1}f$}};
         \node at (-2, 2)  (2HEX2) {\text{\tiny$0$}};
         \node at (-1, -1.628) (2HEX3) {\text{\tiny$0$}} ;
         \node at (-2, 0) (2HEX4) {\text{\tiny$0$}};
         \node at (0, 5.256) (3HEX1) {\text { \tiny$d_{0}d_{1}d_{1}f$}} ;
         \node at (2, 5.256) (3HEX2) {\text { \tiny$d_{0}d_{1}d_{2}f$}};
          \node at (0, -3.256) (4HEX1) {\text{\tiny$0$}} ;
           \node at (2, -3.256) (4HEX2) {\text{\tiny$0$}};
           \node at (-3, 6) (2SQ1) {\emph{\tiny$0$}} ;
          \node at (5, 6) (3SQ1) {\emph{\tiny$0$}};
            \node at (-3, -4) (4SQ1) {\emph{\tiny$0$}} ;
          \node at (5, -4) (5SQ1) {\emph{\tiny$0$}};
                \node at (-3.5, 7) (5HEX1) {\emph{\tiny$0$}} ;
           \node at (5.5, 7) (5HEX2) {\emph{\tiny$0$}};
                 \node at (-3.5, -5) (6HEX1) {\emph{\tiny$0$}} ;
           \node at (5.5, -5) (6HEX2) {\emph{\tiny$0$}};
         \draw (A) -- (B) node[midway,sloped, above] {\text{\tiny$d_{0}d_{0}G_{0}$}};
         \draw (A) -- (C) ;
         \draw (B) -- (D);
         \draw (C) -- (D) node[midway,sloped, below] {\text{\tiny$d_{0}d_{1}G_{0}$}};
         \draw (D) -- (E) node [midway, sloped, below] {\text{\tiny$d_{0}H_{02}f$}};
         \draw (E) -- (F) node [midway, sloped, above] {\text{\tiny$d_{0}d_{0}G_{2}$}};
         \draw (C) -- (G) ;
         \draw (G) -- (H);
         \draw (F) -- (H);
         \draw (B)-- (2HEX1) node [midway, sloped, above] {\text{\tiny$d_{0}H_{12}f$}};
         \draw (2HEX1) -- (2HEX2) node [midway, sloped, above] {\text{\tiny$d_{0}d_{0}G_{1}$}};
         \draw (A) -- (2HEX3);
         \draw (2HEX2) -- (2HEX4);
         \draw (2HEX3) -- (2HEX4);
         \draw (2HEX1) -- (3HEX1) node [midway, sloped, below] {\text{\tiny$H_{01}d_{1}f$}};
         \draw (3HEX1) -- (3HEX2) node [midway, sloped, below] {\text{\tiny$d_{0}H_{02}f$}};
         \draw (E) -- (3HEX2) node [midway, sloped, below] {\text{\tiny$H_{0, 1}d_{2}f$}};
         \draw (2HEX3) -- (4HEX1);
         \draw (4HEX1) -- (4HEX2);
         \draw (G) -- (4HEX2);
         \draw (2HEX2) -- (2SQ1)  ;
         \draw (3HEX1) -- (2SQ1) node [midway, sloped, above] {\text{\tiny$d_{0}d_{1}G_{1}$}};
         \draw (3HEX2) -- (3SQ1) node [midway, sloped, above] {\text{\tiny$d_{0}d_{1}G_{2}$}};
         \draw (F) -- (3SQ1) ;
         \draw (4HEX1) -- (4SQ1);
         \draw (2HEX4) -- (4SQ1);
         \draw (H) -- (5SQ1);
         \draw (4HEX2) -- (5SQ1);
         \draw (2SQ1) -- (5HEX1);
         \draw (3SQ1) -- (5HEX2);
         \draw (5HEX1) -- (5HEX2);
         \draw (4SQ1) -- (6HEX1);
         \draw (5SQ1) -- (6HEX2);
         \draw (6HEX1) -- (6HEX2);
         \draw (5HEX1) -- (6HEX1);
         \draw (5HEX2) -- (6HEX2);

        \node at (1, 1.25) {\text{\tiny$H_{01}G_{0}$}};
        \node at (1, 0.75) {\boxed{\text{\tiny$(d_{0}d_{1})(d_{0}d_{1})$}}};
          \node [rotate = 90] at (3.2, 0.75) {\boxed{\text{\tiny$(d_{0})(d_{0}d_{1}f)$}}};
        \node at (1, -1.75) {\text{\tiny$H_{012} d_{1} = H_{012} 0$}};
        \node at (1, 3.75) {\text{\tiny$H_{012} f$}};
        \node at (1, 3.25)  {\boxed{\text{\tiny$(d_{0}d_{1}d_{2})(d_{0})$}}};
        \node at (-1, 0.75) {$\bar{K}_{0, 1}$} ;
        \node at (1, -4.5) {\boxed{\text{\tiny$d_{0}(0) = 0$}}};
         \node at (1,  6) {\boxed{\text{\tiny$d_{0}(d_{0}d_{2}f)$}}};
        \node at (4.75, 1.8) {\text{\tiny$H_{0, 1, 2}0$}};
        \node at (4.75, 1.2) {\boxed{\text{\tiny$(d_{0}d_{1}d_{2})d_{2}$}}};
        \node at (-2.7, 1.2) {$0$};
        \node at (-1.8, -2.2) {$0$};
        \node at (4, -2.2) {$0$};
        \node at (-1.8, 5) {\text{\tiny$H_{0, 1}G_{1}$}};
        \node at (4, 5) {\text{\tiny$H_{0, 1}G_{2}$}};
        \node at (-1.5, 4.6) {\boxed{\text{\tiny$(d_{0}d_{1})(d_{0}d_{1})$}}};
        \node at (3.7, 4.6) {\boxed{\text{\tiny$(d_{0}d_{1})(d_{0}d_{2})$}}};
\end{tikzpicture}
\end{center}
\caption{Representing $\fd^{3}\lra{f}$ as a map out of $\partial \mP^{3}$}
\label{fopenperm}
\end{figure}

\begin{remark}\label{rmk4.6}
The maps denoted by $H_{ij}$ and $H_{ijk}$ in \S \ref{exam4.4} correspond to $H_{d_{i}d_{j}}$ and $H_{d_{i}d_{j}d_{k}}$ in the notation of \S \ref{con4.5}.
Thus, the map from the boundary of the $3$-permutahedron construction in \S \ref{exam4.4} is precisely $\mT_{d_{0}d_{1}d_{2}d_{3}}$ in the notation of \S \ref{con4.5}.

Note also that the restrictions of $\mT_{\theta}$ to the various facets
of $\partial \mP^{r}$ have a uniform description in terms of decompositions $\theta=\theta_{1}\circ\dotsc\circ\theta_{\ell}$ (corresponding to viewing \eqref{diagnkeq} simply as a truncated restricted simplicial diagram). However, in fact there are three different types of factors involved:
\begin{enumerate}
\renewcommand{\labelenumi}{(\alph{enumi})}
\item The maps $\bmT_{\theta\sb{i}}$ for all but the rightmost factor $\theta_{\ell}$ are the given coherence homotopies for $\xd$, which did not appear in \S \ref{exam4.3} when $\xd$ was a strict simplicial object.
\item If the map $\theta_{\ell}:\ppp{\bnp}\to\pbm$ does not have $f$ (that is,
$d\sp{n+1}\sb{0}:\ppp{\bnp}\to\pbn$, in $\Delta\op$) as a factor, then $\bmT_{\theta\sb{\ell}}=0$.
\item The remaining maps $\bmT_{\theta\sb{\ell}}$, with $\theta\sb{\ell}$ decomposable and having $d\sb{0}\sp{n+1}$ as a factor, are the only ones which also in appear
in \S \ref{exam4.3} (and thus involve choices).
\end{enumerate}
See Figure \ref{fopenperm} for an illustration.
\end{remark}

We can now state our main technical result:

\begin{theorem}\label{thm4.7}
Suppose that we have a diagram $D\colon\DK(\pbDel\op_{n-r+1, n} \times\bbo) \rightarrow \mX$
representing $\lra{f}$ as in \eqref{diagnk}. By the suspension-loop adjunction, the map $\mT_{d_{0}\cdots d_{r}}$  determines an element $\alpha$ of $\pi_{r}\Map_{\mX}(\Sigma^{p}y, x_{n-r})$ representing $\fd^{r}\lra{f}$, and $\alpha$ vanishes if and only if $D$ extends to a diagram $\DK(\pbDel_{n-r, n}\op \times \bbo) \rightarrow \mX$ as in \eqref{diagnk}.
\end{theorem}

To prove this, we require the following:

\begin{lemma}\label{lemb.1}
Let $r \ge 3$, and suppose that $\bmT_{\theta}$ exists for each
$\theta\colon\ppp{\bnp}\rightarrow\ppp{\bn-\br+\bo}$, so that for each
$\theta\colon\ppp{\bnp}\rightarrow\ppp{\bn-\br}$ we can define $\mT_{\theta}$ by \S \ref{con4.5}.
If $\mT_{d_{0} \cdots d_{r}}$ has an extension $\bar{\mT}_{d_{0} \cdots d_{r}}$ to $\mP^{r}$, then   $\bmT_{\theta}$ exists for each $\theta\colon\ppp{\bn+\bo} \rightarrow \ppp{\bn-\br}$.
\end{lemma}

\begin{proof}
As in \S \ref{exam4.3}, we may assume that the simplicial object
$\pbn\mapsto \Map_{X}(\Sigma^{p}y, x_{n})$ is Reedy fibrant (and in particular strict), and proceed by induction on $r$:

For $r = 3$, let $\theta = d_{i}d_{j}d_{k}$ for $i < j < k$. The map $\mT_{d_{i}d_{j}d_{k}}$ is zero if $i > 0$, and therefore trivially admits an extension $\bmT_{d_{i}d_{j}d_{k}}$. Otherwise, the hypothesis that the above diagram is strict means that the map $H_{i, j}$ is the identity, and $\mT_{d_{0}d_{j}d_{k}} = d_{j-1}G_{k-1} \circ (d_{k-2}G_{j-1})^{-1}$. It follows from the Lemma \ref{lem4.1} that $G_{0}$ factors through $C_{n-1}(\xd)$, and $G_{i} = 0$ for all $i > 0$. Thus, the only possible situation where $\mT_{d_{0}d_{j}d_{k}}$ is non-zero is $(i, j, k) = (0, 1, 2)$ and the result follows.

The above discussion implies that $\mT\sb{d_{0}d_{1}d_{2}}$ factors through $C_{n-2}(\xd)$. Lemma \ref{lem4.1} also implies that we can choose $\bar{\mT}_{d_{0}d_{1}d_{2}}$ to factor through $C_{n-2}(\xd)$.

For the general case, suppose that $\bmT_{d_{0} \cdots d_{r-1}}$ factors through $C_{n-r+1}(\xd)$ and for every other map $\theta'\colon\ppp{\bnp}\rightarrow\ppp{\bn-\br+\bo}$ we have $\bmT\sb{\theta} = 0$. Suppose that $\theta\colon\ppp{\bnp}\rightarrow\ppp{\bn-\br}$. Then the only possibly non-zero facets of $\mT_{\theta}$ are those corresponding to factorizations $(\theta' )(d_{0} \cdots d_{k})$ of $\theta$. Since we have a strict simplicial object, the corresponding facet of $\mT_{\theta}$ is $\theta' \circ \bar{\mT}_{d_{0}, \cdots d_{k}}$.

 We can write $\theta = d_{i_{1}} d_{i_{2}} \cdots d_{i_{l}}$, with $i_{1} < i_{2} \cdots < i_{l}$ and we can show inductively that $\bar{\mT}_{d_{0} \cdots d_{k}}$ factors through $C_{n-k+1}(\xd)$. Thus, $\mT_{d_{0}d_{1} \cdots d_{r}}$ is the only such map which is non-zero on some face; that is, the face corresponding to $(d_{0})(d_{0} \cdots d_{r-1})$. The result follows.
\end{proof}

\begin{corollary}\label{corb.2}
An element $\lra{f}\in E^{1}\sb{n,p}$ represented by $f\colon\Sigma^{p}y \rightarrow x_{n}$ survives to the $E^{r}$-page of the spectral sequence if $\mT_{d_{0} \cdots d_{r-1}}$ is nullhomotopic, and $\fd^{r}\lra{f}$ is then represented by $\mT_{d_{0} \cdots d_{r}}$
\end{corollary}

\begin{proof}
We again reduce to the Reedy fibrant case and proceed by induction. The cases $r = 2, 3$ follow from \S \ref{exam4.4} and \S \ref{rmk4.6}. In general, suppose that $\mT_{d_{0}\cdots d_{r-1}}$ is identified with $\fd^{r-1}\lra{f}$. Then by the general case in the proof of Lemma \ref{lemb.1}, $\mT_{d_{0} \cdots d_{r}}$ can be identified with $d_{0}\bmT_{d_{0}, \cdots d_{r-1}}$, and $\bmT_{d_{0}, \cdots d_{r-1}}$ can be identified with a choice of nullhomotopy of $\mT_{d_{0} \cdots d_{r-1}}$, so the result follows by the last paragraph of \S \ref{con4.2}.
\end{proof}

\begin{proof}[Proof of Theorem \ref{thm4.7}]
By induction on $r$, starting with the trivial case  $r = 2$:

Suppose that we have defined $\cDD\colon\DK(\pbDel_{n-r+1, n}\op\times\bbo) \rightarrow \sSet$ as in Theorem \ref{thm3.5}.
For each morphism $\theta\colon(\pbm, a) \rightarrow (\pbk, b)$ in $\pbDel_{n-r+1, n}\op\times\bbo$ (with $a,b\in\{0,1\}$), we write $C_{\theta}$ for the corresponding component of
$$
\DK(\pbDel_{n-r+1, n}\op \times\bbo) ( (\pbm, a), (\pbk, b))~,
$$
and $\bar{C}_{\theta}$ for the component of $\DK(\pbDel\op)(\pbm,\pbk)$ corresponding to $\theta\colon\pbm\rightarrow\pbk$. We define a map $\sk_{r-1}\DK(\pbDel_{n-r, n}\op \times\bbo)\rightarrow\sSet$ extending $\cDD$ by
$$
C_{\theta} \mapsto
\begin{cases}
        0  & \mbox{if } a = b = 0, m < n \\
        \xd(\bar{C}_{\theta})  & \mbox{if } a = b = 1 \\
        \cDD(C_{\theta}) & \mbox{if } k > r \\
        \mT_{\theta' \circ d_{0}} & \mbox{if } \theta =
      (\theta', 0 \rightarrow 1)\colon(\pbn, 0) \rightarrow(\ppp{\bn-\br}, 1).
\end{cases}
$$
Note that the simplicial $\bHom$-sets of $\DK(\pbDel_{n-r, n}\op \times\bbo)$ are $(r-1)$-skeletal, except for
$$
\DK(\pbDel_{n-r, n}\op \times\bbo)((\pbn,0), (\ppp{\bn-\br}, 1))~.
$$
This is a disjoint union of $r$-permutahedra corresponding to the maps $\theta\colon(\pbn, 0) \rightarrow (\ppp{\bn-\br}, 1)$. The result follows from Lemma \ref{lemb.1} and Corollaries \ref{corb.2} and \ref{cora.5}.
\end{proof}

\begin{corollary}\label{cor4.8}
The value of $\fd^{r}\lra{f}$ associated to each
$D\colon\DK(\pbDel\op_{n-r+1, n} \times\bbo) \rightarrow \mX$ as in Theorem \ref{thm4.7} is
determined inductively by expressing $\Sigma^{p}y\wedge\partial\mP^{r}$ as a colimit in $\mX$
of a certain diagram $\ccL$, and using its universal property on the inductively-defined
constituents of this diagram.
\end{corollary}

\begin{proof}
By construction, every simplicial set $K$ is the (homotopy) colimit of the tautological
diagram $\cG:\Delta_{/K}\op\to\sSet$ indexed by its simplex category (see \S \ref{def2.4}),
so a map $f:K\to L$ is determined up to homotopy by the compatible collection of maps
$f\sb{\sigma}:\sigma\to L$ for $\sigma\in\Delta_{/K}\op$.
However, \eqref{DKmapping2} displays $\partial\mP^{r}$ inductively as a (homotopy) colimit of
a simpler diagram $\ccH:\cF\sb{\mP^{r}}\to\sSet$, where $F\sb{\mP^{r}}$ is now the category of
polyhedral faces of $\mP^{r}$, and by the dual of \cite[XI, 4.3]{BK}, the
homotopy colimit of $\cG\colon\Delta\op_{/\partial\mP^{r}}\to\sSet$ can be obtained by first taking the homotopy colimit over the simplex category of each (product of) lower dimensional permutahedra in $\cF\sb{\mP^{r}}$, and then taking the homotopy colimit over $\cF\sb{\mP^{r}}$.
Conversely, for each $K\in\cF\sb{\mP^{r}}$, any map $f\sb{K}:\ccH(P)\to L\in\sSet$ determines
maps $f\sb{\sigma}:\sigma\to L$ for each $\sigma\in\Delta_{/K}\op$, and thus we can obtain
the induced map $f:\partial \mP^{r}\to L$ from a compatible collection $(f\sb{K})\sb{K\in\cF\sb{\mP^{r}}}$, by the universal property of the (homotopy) colimit over
$\cF\sb{\mP^{r}}$.

In particular, this holds for the map $f=\mT_{\theta}\colon\partial\mP^{r} \rightarrow \Map_{\mathscr{X}}(\Sigma^{p}y, x_{n-r})$, which is induced inductively from the fill-ins $\bmT_{\theta'}$,
as in Lemma \ref{lemb.1}, together with the nullhomotopies $G_{k}$ of \S \ref{exam4.3}
and the coherence homotopies $H\sb{I}$ for the simplicial object $\xd$ (which are naturally described as maps from suitable permutahedra to $\Map_{\mX}(x\sb{n}, x\sb{n-i})$, by Remark \ref{rmk1.7} and Proposition \ref{lema.3}).

By \S \ref{con4.5}, the map $\mT_{\theta}$ is pointed (where the basepoint of
$\partial\mP^{r}$ is the vertex labelled $d_{i_{0}}\cdots d_{i_{r}}$ with ascending indices),
so it is adjoint to a map $T_{\theta}\colon\Sigma^{p}y\wedge\partial\mP^{r}\rightarrow x_{n-r}$.
Here for any $z\in \Ob( \mX)$ and simplicial set $K$ with basepoint $k$ we define $z\wedge K$ to be
the (homotopy) cofiber of $z\otimes\{k\}\hra z\otimes K$.
Note that the simplicial tensoring with $K$ is defined by taking a (homotopy) colimit in $\mX$
of the constant $z$-diagram indexed by $\Delta_{/K}\op$, and by the discussion above
when $K=\partial\mP^{r}$ we may replace $\Delta_{/K}\op$ by $\cF\sb{\mP^{r}}$ (because
products of permutahedra, like simplices, are contractible).

This allows us to write $\Sigma^{p}y\wedge\partial\mP^{r}$ as a colimit in $\mX$ (see \S \ref{def2.3} and Remark \ref{rmk2.8}) of a diagram $\ccL\colon\cF\sb{\mP^{r}}\to \mX$,
with the values at each object being either $\Sigma^{p}y$ or $0$.
The class $[T_{\theta}]\in\pi_{0}\Map_{\mX}(\Sigma^{p+r-1}y, x_{n-r})$ is then determined
inductively by the given diagram $\ccL$ using this description of its domain as a
colimit in $\mathscr{X}$.
\end{proof}

\begin{remark}\label{rmk4.9}
  In any reasonable models of $\infty$-category theory, the (co)limits in an $\infty$-category $X$ can be identified with the homotopical (co)limits in the corresponding strict model $\mX$ (see \S \ref{rmk1.9}). We can also identify the
homotopy spectral sequences for simplicial objects in $X$ and
in $\mX$. In Section \ref{cqcic}, we will see how both these identifications work precisely
in a quasi-category.

Thus, Corollary \ref{cor4.8} provides a model-independent description of the spectral
sequence, for any of the standard models of $\infty$-category theory
(see \cite{BERGNER3,Bergner-Arising}).
\end{remark}

\begin{remark}\label{rperm}
Permutahedra have appeared in a number of different contexts in homotopy theory,
starting with \cite{MilgI} and \cite{BauG}, and in particular in connection with 
higher homotopy operations (see \cite{BlaA,BJTurHA,BMarkH}) where they were used 
as an \emph{ad hoc} formalism for keeping track of the homotopies needed to define such
operations. However, the connection with the Dwyer-Kan resolution of $\pbDel$, which may
explain their central role in various settings, was not noticed previously
(by the first author, at least).
\end{remark}

\begin{mysubsection}{Differentials in stable $\infty$-categories}
\label{sdsic}
At first glance one might think that little of the above description is needed in a stable
$\infty$-category $X$, where spectral sequences are constructed directly from a tower of
(co)fibrations
$$
\dotsc \to x\sb{n}~\xrightarrow{f\sb{n}}~x\sb{n-1}~\xrightarrow{f\sb{n-1}}~\dotsc ~\xrightarrow{f\sb{2}}~x\sb{1}~\xrightarrow{f\sb{1}}~x\sb{0}~.
$$
However, in fact in this case too a restricted simplicial diagram is needed to keep track of
all the coherences.   This is illustrated in the case $n=3$ by the cubical diagram:
\myudiag[\label{tgencube}]{
  x\sb{3} \ar[rr] \ar[dd] \ar[rd]^(0.65){f\sb{3}} && 0 \ar@{->}|(.5){\hole}[dd] \ar[rd] &\\
  &  x\sb{2} \ar[rr] \ar[dd]_(0.3){f\sb{2}} && 0 \ar[dd] \\
  0 \ar@{->}|(.5){\hole}[rr] \ar[rd] && 0 \ar[rd] &\\
  & x\sb{1} \ar[rr]^{f\sb{1}} && x_{0}
}
(which is just \eqref{diagnkeq} with all but the first face map zero in
\emph{all} dimensions). See \cite{BBSenT} for further details.
\end{mysubsection}

%
%
\sect{Quasi-Categories}
\label{cqcic}

We now turn to another model of $\infty$-category theory, originally due to
Boardman-Vogt in \cite{BVogHI}, and later studied by Joyal and Lurie (see
\cite{Joyal1,Joyal-quasi-cat,Lurie}):

\begin{defn}\label{def5.1}
A \emph{quasi-category} is a simplicial set $X$ in which we can find all fillers
$$
\xymatrix
{
\Lambda_{i}^{n} \ar[r] \ar[d] & X \\
\Delta^{n} \ar@{.>}[ur] &
 }
$$
for $0< i < n$.
\end{defn}

The category of simplicial sets admits a \emph{Joyal} model category structure, in which the cofibrations are monomorphisms, the fibrant objects are quasi-categories, and
the weak equivalences are \emph{Joyal equivalences} (see \cite[\S 2.2]{Lurie} or \cite{Joyal-quasi-cat}). Moreover:

\begin{theorem}[\protect{\cite[Theorem 2.2.5.1]{Lurie}}]\label{thm5.2}
There is a Quillen equivalence:
$$
\fC\colon\sSet\leftrightarrows\sCat\colon\fB
$$
between the Bergner model structure of \S \ref{exam1.5} and the Joyal model structure. The right adjoint is known as the \emph{homotopy coherent nerve}.
\end{theorem}

\begin{defn}\label{def5.4}
Given a quasi-category $X$ and $x, y \in X_{0}$, we denote the pullback of
$\ast\xrightarrow{x}  X \leftarrow X_{/y}  $ by $\Map_{X}(x, y)$.

The \emph{join functor} is the unique cocontinuous functor $(-) * (-)\colon\sSet \times \sSet \rightarrow \sSet$, such that $\Delta^{n} * \Delta^{m}
\cong \Delta^{n+m+1}$ (the usual join).
\end{defn}

\begin{defn}\label{def5.5}
  We write $B\colon\Cat\rightarrow \sSet$ for the usual nerve functor. This functor has a left adjoint $\cP\colon\sSet\rightarrow\Cat$, where if $X$ is a quasi-category, its \emph{homotopy category} $\cP(X)$ has as objects the vertices of $X$, with
  $\cP(x, y) = \pi_{0}\Map_{X}(x, y)$. The equivalence class of $f\in\Map_{X}(x, y)_{0}$ is denoted by $[f]$.

Given $x \xrightarrow{f} y \xrightarrow{g} z$, we can find a lift
$$
\xymatrix
{
\Lambda_{1}^{2} \ar[d] \ar[rr]_{(g, \,, f)} && X \\
\Delta^{2} \ar@{.>}[urr]_{s} &&
}
$$
and put $[g] \circ [f] = [d_{1}s]$. The composition operation in $\cP(X)$ is independent of the choice of lift (see \cite[1.2.4]{Lurie}).

A $2$-simplex $\sigma$ of a quasi-category $X$ with boundary
$$
\xymatrix
{
&  1 \ar[dr]^{g} & \\
0 \ar[ur]^{f} \ar[rr]_{h} & & 2
}
$$
is said to \emph{witness} the composition $g \circ f = h$.
More generally, an $n$-simplex $\sigma$ in $X$ \emph{witnesses} the composition of the $1$-simplices
$$
0 \xrightarrow{e_{0}} 1 \xrightarrow{e_{1}} 2 \cdots \xrightarrow{e_{n-1}} n.
$$
\end{defn}

\begin{mysubsection}{Quasi-categories as $(\infty, 1)$-categories}
\label{sqcic}
We can view quasi-categories as a model for $(\infty, 1)$-category theory by taking the vertices of the quasi-category $X$ for $\Ob(X)$, and the space of maps between two objects to be the Kan complex of \S \ref{def5.4}. The nerve functor is the usual nerve functor and the homotopy category functor is as in \S \ref{def5.5}.

There is a natural zig-zag of weak equivalences between $\Map_{X}(x, y)$ and
$\fC(X)(x, y)$ (see \cite[\S 2.2.2]{Lurie}). Thus, we can take our composition operation $c_{x, y, z}$ to be any lift
$$
\xymatrix
{
\Map_{X}(x, y) \times \Map_{X}(y, z) \ar[d] \ar@{.>}[rr] &&  \Map_{X}(x, z) \ar[d] \\
\fC(X)(x, y) \times \fC(X)(y, z) \ar[rr]_{c}  && \fC(X)(x, z)
}
$$
in the homotopy category of spaces, with the bottom map being composition in $\fC(X)$. Associativity up to homotopy is immediate, since composition in $\fC(X)$ is strictly associative.

The Joyal equivalences of quasi-categories are precisely equivalences of $\infty$-categories in the sense of \S \ref{def1.3} (\cite[Theorem 2.2.0.1]{Lurie}).
Therefore, we can think of $\fC(X)$ as a strict model for the
$\infty$-category $X$. In fact, this is precisely the model for
the $\infty$-category $X$ given by \S \ref{rmk1.9}, up to weak
equivalence.

Consider the Quillen equivalence $q_{M}\colon\sSet
\leftrightarrows \stSet\colon r_{M} $ between the Joyal and complete
Segal model structure, with the right adjoint $r_{M}$ given by
$S \mapsto S_{0,\ast}$.
By \cite{BERGNER3}, the homotopy inverse to the functor $E$ assigning to any quasi-category its strict model (given in \S \ref{rmk1.9}) is $\bR r_{M} \circ \bL I \circ \bR B$, where $B$ is the nerve functor applied levelwise to a simplicial category, and $I$ is the inclusion of Segal precategories in bisimplicial sets. As noted in the final paragraph of \cite[Section 3]{Bergner-Arising} $\bL I \circ \bR B$ is weakly equivalent to $\bL q_{M} \circ \bR\fB$, so that $\bR r_{M} \circ \bL I \circ \bR B$ is weakly equivalent to $\bR\fB$.
Thus on the level of homotopy categories the functor $E$ is isomorphic to $\bL\fC$.
\end{mysubsection}

%
%
\sect{Differentials of spectral sequences in a quasi-category}
\label{cdssqc}

We now show that in a quasi-category $X$, the spectral sequence of \S \ref{con2.6} can be identified with the spectral sequence in an associated simplicial category
(provided $X$ has enough (co)limits). We also give an explicit, intrinsically quasi-categorical, analogue of Theorem \ref{thm4.7} for the first few differentials.

\begin{lemma}\label{lem6.1}
If $X$ is a quasi-category with countable (co)limits, then simplicial objects in $X$ have
homotopy spectral sequences. Conversely, if $\mX$ is a fibrant simplicial category with countable (co)limits, and $y$ and $\xd$ are respectively a cogroup object and a simplicial object in $\mX$, then the spectral sequence of \S \ref{con3.3} in $\mX$ is isomorphic to that associated to $\fB(\xd)$ in $\fB(\mX)$.
\end{lemma}

\begin{proof}
For the first statement, it suffices to show that hypotheses \S \ref{con2.6}(1),(2) hold in $X$: the derived counit of the Quillen equivalence of Theorem \ref{thm5.2} yields a Joyal equivalence $X \rightarrow \fB(\mX)$ for some fibrant simplicial category $\mX$. By \cite[Prop.~4.1.1.9]{Lurie}, we can identify colimits in $X$ with those in $\fB(\mX)$, so it suffices to check the conditions in $\fB(\mX)$. By \cite[Lemma 6.5]{Riehl1}, simplicial objects in $\fB(\mX)$ are the same as homotopy coherent simplicial objects in $\mX$, so we may replace $\xd$ with $\fB(\zd)$ for some simplicial object $\zd$ in $\mX$ (see also \cite[Cor.~4.2.4.7]{Lurie}).

The tensoring $\xd[K]$ can be identified with the tensoring $\zd[K]$ using \cite[Theorem 4.2.4.1]{Lurie} (which identifies colimits in $\mX$ with colimits in $\fB(\mX)$). Thus, condition (1) follows from the corresponding statement for simplicial categories.

For condition (2), by \cite[Cor.~4.2.4.7]{Lurie}, we may replace any fibre sequence in $\fB(\mX)$ with a fibre sequence $a \rightarrow b \rightarrow c$ in $\mX$.  By \cite[2.2.2]{Lurie} there is then a natural zig-zag of weak equivalences connecting $\Map_{\fB(\mX)}(x, y)$ to $\Map_{\mX}(x, y)$. Thus,
 $$
 \Map_{\fB(\mX)}(y, a) \rightarrow \Map_{\fB(\mX)}(y, b) \rightarrow \Map_{\fB(\mX)}(y, c)
 $$
is a fibre sequence, since it is identified (up to homotopy equivalence) with
 $$
 \Map_{\mX}(y, a) \rightarrow \Map_{\mX}(y, b) \rightarrow \Map_{\mX}(y, c),
 $$
 yielding condition (2).

The above discussion also yields a vertical weak equivalence
$$
\xymatrix{
\Map_{\fB(\mX)}(\Sigma^{p}y, Z_{n}(\xd )) \ar[r] \ar[d] & \Map_{\fB(\mX)}(\Sigma^{p}y, C_{n}(\xd )) \ar[r] \ar[d] & \Map_{\fB(\mX)}(\Sigma^{p}y, Z_{n-1}(\xd )) \ar[d] \\
  \Map_{\mX}(\Sigma^{p}y, Z_{n}(\zd )) \ar[r] & \Map_{\mX}(\Sigma^{p}y, C_{n}(\zd )) \ar[r] & \Map_{\mX}(\Sigma^{p}y, Z_{n-1}(\zd ))
      }
$$
of the fibre sequences which define the spectral sequences of \S \ref{con2.6} and \S \ref{con3.3}, so the final statement of the Lemma follows.
\end{proof}

\begin{defn}\label{dlift}
We say that $f\colon\cB(C) \rightarrow X$ \emph{lifts} a diagram $D\colon C \rightarrow \cP(X)$ if $\cP(f)$ is isomorphic to $D$.
\end{defn}

\begin{lemma}\label{lem6.2}
Given a diagram $D\colon C \rightarrow \cP(X)$ and a Joyal equivalence of quasi-categories
$g\colon X \rightarrow Y$, there is a lift $h$ for  $\cP(g) \circ D$ if and only if $D$ has a lift $f$.
\end{lemma}

\begin{proof}
If $f\colon\cB(C) \rightarrow X$ is a vertex of $X^{\cB(C)}$, then $\cP(f)$ can be identified with the image of $f$ under the map
 $$
 X^{\cB(C)} \rightarrow\cB\cP(X)^{\cB(C)}
 $$
 induced by the counit of the adjunction $\cP \dashv\cB$. There is a functor
 $J:\sSet\to\Kan$ which assigns to each simplicial set its maximal Kan subcomplex (see \cite[Corollary 1.5]{Joyal1}), and $D$ has a lift if and only if the component of $D$ is contained in the image of
 $$
 \pi_{0}J(X^{\cB(C)}) \rightarrow \pi_{0}J(\cB\cP(X)^{\cB(C)} ) \cong \pi_{0}J\cB((\cP(X)^{C})) \cong \pi_{0}\cB(\Iso(\cP(X)^{C})),
 $$
where $\Iso(\cK)$ is the subcategory of natural isomorphisms in a category $\cK$.
  We have a commutative diagram
 $$
 \xymatrix
 {
  \pi_{0}J(X^{\cB(C)}) \ar[r] \ar[d] &  \pi_{0}J(\cB\cP(X)^{\cB(C)} ) \ar[d] \\
  \pi_{0}J(Y^{\cB(C)}) \ar[r] &  \pi_{0}J(\cB\cP(Y)^{\cB(C)} )
 }
 $$
 The functor $(-)^{\cB(C)}$ preserves Joyal equivalences of quasi-categories, and $J$ takes Joyal equivalences to weak equivalences (see \cite[Lemma 4.5]{Joyal-quasi-cat}).
 The vertical arrows are thus bijections, and the required result follows.
\end{proof}

\begin{prop}\label{thm6.3}
Let $\xd$ be a simplicial object in a quasi-category $X$, and $y$ a cogroup object of $X$. Then $\lra{f}\in \pi_{0}\Map_{X}(\Sigma^{p}y,x\sb{n})$ survives to the $E^{r}$-page of the spectral sequence of \S \ref{con2.6} if and only if we can find a diagram in $X$ as in \eqref{diagnk}.
\end{prop}

\begin{proof}
The class $\lra{f}$ survives to $E\sp{2}\sb{n,p}$ if and only if
$\lra{f}\in C\sb{n}\pi_{0}\Map_{X}(\Sigma^{p}y,\xd)$, in which case we have a diagram
$\widetilde{D}:\DK(\pbDel\op_{n-r, n} \times\bbo)\to\cP(X)$, which we would like to lift to $X$.
Using Lemma \ref{lem6.2}, we can reduce to the case that $X = \fB(\mX)$ for some fibrant simplicial category $\mX$. By \cite[Lemma 6.5]{Riehl1} a diagram
$\cB(\pbDel_{n-r, n}\op) \rightarrow \fB(\mX)$ is equivalent to a homotopy coherent
diagram $\DK(\pbDel_{n-r, n}\op) \rightarrow \mX$. Furthermore, by Lemma \ref{lem6.1} the spectral sequence for $y$ and $\xd$ is isomorphic to the spectral sequence for the
corresponding simplicial object in $\mX$. Thus, the statement follows from Theorem \ref{thm3.5}.
\end{proof}

\begin{remark}\label{con6.4}
Suppose that $d\colon\Delta^{n} \rightarrow X$ is a diagram in a pointed quasi-category $X$ witnessing an $n$-fold composition. We note that the map $X_{/0} \rightarrow X $ is a trivial fibration (\cite[Proposition 1.2.12.4]{Lurie}). Thus, we can find an extension $\gamma\colon\Delta^{n+1} \cong \Delta^{0} * \Delta^{n} \rightarrow X$ with the vertex of $\Delta^{0}$ mapping to the zero object. Let us choose a $1$-simplex
  $s\colon y \rightarrow 0$. Then $\gamma$ and $s$ define a map
$$
\xymatrix{
  \Delta^{1} \cup_{\Delta^{0}} \Delta^{n+1} \ar[r] \ar[d] & X \\
   \Delta^{n+2} \ar@{.>}[ur]_{\sigma} &
   }
$$
The vertical map is inner anodyne by \cite[Lemma 3.3]{Riehl1}, so we can find the lift $\sigma$. We can think of the face $\sigma \circ d^{1}$ as a \emph{canonical witness}
$f_{n} \circ \cdots \circ f_{1} \circ 0$, which must exist because of the universal property of the zero map.
\end{remark}

\begin{example}\label{exam6.5}
  Suppose that we want to determine whether $\lra{f}\in\pi_{0}C_{1}\Map_{X}(\Sigma^{p}y,\xd)$ survives to $E\sp{\infty}\sb{1,p}$, and thus
  represents a vertex in $\Map_{X}(\Sigma^{p+1}y,\|\xd\|)$. For this we need a lift of the diagram
\myudiag[\label{eqexphexagon}]{
\Sigma^{p}y~~ \ar@<2ex>[r]^{f = d_{0}} \ar[r]^{0 = d_{1}} \ar@<-2ex>[r]^{0 = d_{2}} & x_{1}\ar@<1ex>[r]^{d_{0}} \ar@<-1ex>[r]_{d_{1}} & x_{0} \ar[r]_{d_{0}} & \|\xd\|
}
in $\cP(X)$ (compare \eqref{diagnkeq}) to the diagram in Figure \ref{fqctwoper}:

\stepcounter{theorem}
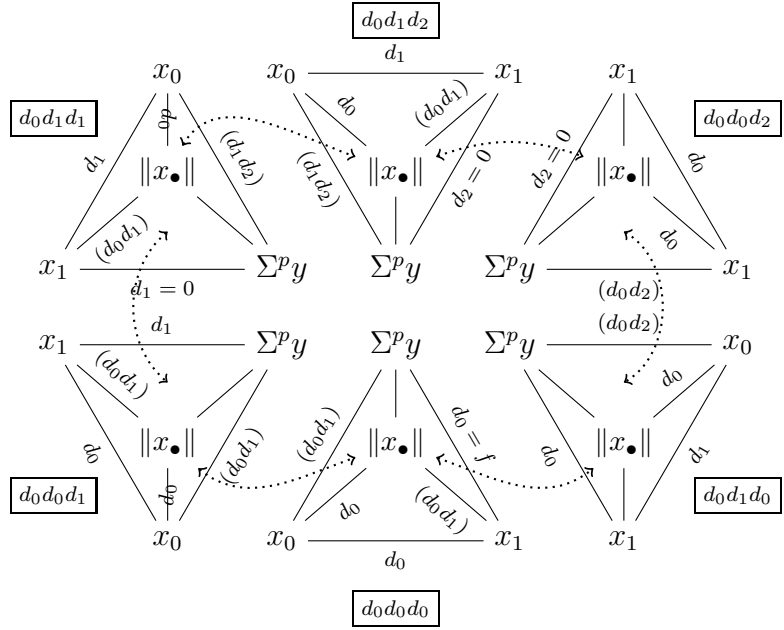
\begin{figure}[htbp]
\begin{center}
\begin{tikzpicture}
                 \node at (3, 0) (1vert3) {$\Sigma^{p}y$};
                    \node  at (0, 0) (1vert2) {$x_{1}$};
                 \node at (1.5, 2.6) (1vert1) {$x_{0}$};
                 \node at (1.5, 1.3) (1vert0) {\text{\small$\|\xd\|$}};
                \node at (3, -1) (2vert3) {$\Sigma^{p}y$};
                \node at (0, -1) (2vert2) {$x_{1}$};
                \node at (1.5, -3.6) (2vert1) {$x_{0}$};
                \node at (1.5, -2.3) (2vert0) {\text{\small$\|\xd\|$}};

                \node at (4.5, -1) (3vert3) {$\Sigma^{p}y$};
                 \node at (6, -3.6) (3vert2) {$x_{1}$};
               \node at (3, -3.6) (3vert1) {$x_{0}$} ;
               \node at (4.5, -2.3) (3vert0) {\text{\small$\|\xd\|$}};

               \node at (4.5, 0) (4vert3) {$\Sigma^{p}y$};
               \node at (6, 2.6) (4vert2) {$x_{1}$};
               \node at (3, 2.6) (4vert1) {$x_{0}$};
               \node at (4.5, 1.3) (4vert0) {\text{\small$\|\xd\|$}};

                 \node at (6, 0) (5vert3) {$\Sigma^{p}y$};
                  \node at (7.5, 2.6) (5vert2) {$x_{1}$};
                    \node  at (9, 0) (5vert1) {$x_{1}$};
                 \node at (7.5, 1.3) (5vert0) {\text{\small$\|\xd\|$}};
                \node at (6, -1) (6vert3) {$\Sigma^{p}y$};
                \node at (7.5, -3.6) (6vert2) {$x_{1}$};
         \node at (9, -1) (6vert1) {$x_{0}$};
                \node at (7.5, -2.3) (6vert0) {\text{\small$\|\xd\|$}};

                 \draw (1vert3) -- (1vert2) node [midway, sloped, below] {\text{\tiny$d_{1} = 0 $}};
                \draw (1vert3) -- (1vert1) node [midway, sloped, above] {\text{\tiny$(d_{1}d_{2})$}};;
                \draw (1vert3) -- (1vert0);
                \draw (1vert2) -- (1vert1) node [midway, sloped, above] {\text{\tiny$d_{1}$}};
                \draw (1vert2) -- (1vert0) node [midway, sloped, below] {\text{\tiny$(d_{0}d_{1})$}};
                \draw (1vert1) -- (1vert0) node [midway, sloped] {\text{\tiny$d_{0}$}};
                \draw (2vert3) -- (2vert2) node [midway, sloped, above] {\text{\tiny$d_{1}$}};
                \draw (2vert3) -- (2vert1) node [midway, sloped, below] {\text{\tiny$(d_{0}d_{1})$}};
                \draw (2vert3) -- (2vert0);
                \draw (2vert2) -- (2vert1) node [midway, sloped, below] {\text{\tiny$d_{0}$}};
                \draw (2vert2) -- (2vert0) node [midway, sloped, above] {\text{\tiny$(d_{0}d_{1})$}};
                \draw (2vert1) -- (2vert0) node [midway, sloped] {\text{\tiny$d_{0}$}};
                \draw (3vert3) -- (3vert2) node [midway, sloped, above] {\text{\tiny$d_{0} = f$}};
                \draw (3vert3) -- (3vert1)  node [midway, sloped, above] {\text{\tiny$(d_{0}d_{1})$}};
                \draw (3vert3) -- (3vert0);
                \draw (3vert2) -- (3vert1) node [midway, sloped, below] {\text{\tiny$d_{0}$}};
                \draw (3vert2) -- (3vert0) node [midway, sloped, below] {\text{\tiny$(d_{0}d_{1})$}};
                \draw (3vert1) -- (3vert0) node [midway, sloped, below] {\text{\tiny$d_{0}$}};
                \draw (4vert3) -- (4vert2) node [midway, sloped, below] {\text{\tiny$d_{2} = 0$}};
                \draw (4vert3) -- (4vert1) node [midway, sloped, below] {\text{\tiny$(d_{1}d_{2})$}};
                \draw (4vert3) -- (4vert0);
                \draw (4vert2) -- (4vert1) node [midway, sloped, above] {\text{\tiny$d_{1}$}};
                \draw (4vert2) -- (4vert0) node [midway, sloped, above] {\text{\tiny$(d_{0}d_{1})$}};
                \draw (4vert1) -- (4vert0) node [midway, sloped, above] {\text{\tiny$d_{0}$}};
                \draw (5vert3) -- (5vert2) node [midway, sloped, above] {\text{\tiny$d_{2} = 0$}};
                \draw (5vert3) -- (5vert1) node [midway, sloped, below] {\text{\tiny$(d_{0}d_{2})$}};
                \draw (5vert3) -- (5vert0);
                \draw (5vert2) -- (5vert1) node [midway, sloped, above] {\text{\tiny$d_{0}$}};
                \draw (5vert2) -- (5vert0) ;
                \draw (5vert1) -- (5vert0) node [midway, sloped, below] {\text{\tiny$d_{0}$}};

                \draw (6vert3) -- (6vert2) node [midway, sloped, below] {\text{\tiny$d_{0}$}};
                \draw (6vert3) -- (6vert1) node [midway, sloped, above] {\text{\tiny$(d_{0}d_{2})$}};
                \draw (6vert3) -- (6vert0);
                \draw (6vert2) -- (6vert1) node [midway, sloped, below] {\text{\tiny$d_{1}$}};
                \draw (6vert2) -- (6vert0) ;
                \draw (6vert1) -- (6vert0) node [midway, sloped, above] {\text{\tiny$d_{0}$}};

        \draw [<->, dotted, thick] (1.5, 0.5)  to [out=-140, in=140] (1.5, -1.5);
        \draw [<->, dotted, thick] (2vert0)  to [out=-40, in=200] (3vert0);
        \draw [<->, dotted, thick] (3vert0)  to [out=-20, in=220] (6vert0);
        \draw [<->, dotted, thick] (7.5, 0.5)  to [out=-30, in=30] (7.5, -1.5);
        \draw [<->, dotted, thick] (1vert0)  to [out=60, in=160] (4vert0);
        \draw [<->, dotted, thick] (4vert0)  to [out=20, in=160] (5vert0);

        \node at (0, 2) {\boxed{\text{\tiny$d_{0}d_{1}d_{1}$}}};
        \node at (0, -3) {\boxed{\text{\tiny$d_{0}d_{0}d_{1}$}}};
        \node at (4.5, 3.3) {\boxed{\text{\tiny$d_{0}d_{1}d_{2}$}}};
        \node at (9, 2) {\boxed{\text{\tiny$d_{0}d_{0}d_{2}$}}};
        \node at (9, -3) {\boxed{\text{\tiny$d_{0}d_{1}d_{0}$}}};
        \node at (4.5, -4.5) {\boxed{\text{\tiny$d_{0}d_{0}d_{0}$}}};
\end{tikzpicture}
\end{center}
\caption{Quasi-category version of the $2$-permutahedron}
\label{fqctwoper}
\end{figure}

This consists of six $3$-simplices (witnesses for the compositions $d_{i}d_{j}d_{k}$ in  \eqref{eqexphexagon}), with faces glued together along the dotted arrows.
Note that the four labelled $d_{i}d_{j}d_{k}$ with $k > 0$
automatically have fillers, by \S \ref{con6.4}. Thus, $\lra{f}$ survives if and only
if we can find a filler for the union of two $3$-simplices with boundary:

\begin{center}
\begin{tikzpicture}{
%
                \node at (4.5, -1) (3vert3) {$\Sigma^{p}y$};
                 \node at (6, -3.6) (3vert2) {$x_{1}$};
               \node at (3, -3.6) (3vert1) {$x_{0}$} ;
               \node at (4.5, -2.3) (3vert0) {\text{\small$\|\xd\|$}};

                \node at (6, -1) (6vert3) {$\Sigma^{p}y$};
                \node at (7.5, -3.6) (6vert2) {$x_{1}$};
         \node at (9, -1) (6vert1) {$x_{0}$};
                \node at (7.5, -2.3) (6vert0) {\text{\small$\|\xd\|$}};

                \draw (3vert3) -- (3vert2) node [midway, sloped, above] {\text{\tiny$d_{0}$}};
                \draw (3vert3) -- (3vert1) node [midway, sloped, above] {\text{\tiny$(d_{0}d_{2})$}};
                \draw (3vert3) -- (3vert0);
                \draw (3vert2) -- (3vert1) node [midway, sloped, below] {\text{\tiny$d_{1}$}};
                \draw (3vert2) -- (3vert0) ;
                \draw (3vert1) -- (3vert0) node [midway, sloped, above] {\text{\tiny$d_{0}$}};

                \draw (6vert3) -- (6vert2) node [midway, sloped, below] {\text{\tiny$d_{0}$}};
                \draw (6vert3) -- (6vert1) node [midway, sloped, above] {\text{\tiny$(d_{0}d_{2})$}};
                \draw (6vert3) -- (6vert0);
                \draw (6vert2) -- (6vert1) node [midway, sloped, below] {\text{\tiny$d_{0}$}};
                \draw (6vert2) -- (6vert0) ;
                \draw (6vert1) -- (6vert0) node [midway, sloped, above] {\text{\tiny$d_{0}$}};

        \draw [<->, dotted, thick] (3vert0)  to [out=-20, in=220] (6vert0);
        }
\end{tikzpicture}
\end{center}

Again, all faces but the bottom two are either `coherences' which are part of the data for $\xd$, or are `canonical' faces witnessing composition with the zero map. Thus, whether or not we can find an appropriate filler $\partial \Delta^{3} \cup_{\Delta^{2}} \partial \Delta^{3} \rightarrow X$ (and thus show that $\lra{f}$ survives) is determined by a choice of $2$-simplices
$$
\xymatrix
{
 & x_{1} \ar[dr]^{d_{0}} & & & \ar[dr]^{d_{1}} x_{1} & \\
\Sigma^{p}y \ar[ur]^{f} \ar[rr] \ar@{}[urr]_{G_{0}} &  & x_{0} & \Sigma^{p}y \ar[ur]^{f} \ar@{}[urr]_{G_{1}} \ar[rr] &  & x_{0}
}
$$
witnessing the composition. We can think of this as the `indeterminacy' in trying to rectify \eqref{eqexphexagon}.
\end{example}

\begin{example}\label{exam6.6}
Next, we want to know whether  $\lra{f}\in\pi_{0}C_{2}\Map_{X}(\Sigma^{p}y,\xd)$ survives to $E\sp{\infty}\sb{2,p}$. In the quasi-category $X$ we have the boundary of six $4$-simplices glued together along the edges of the hexagon, which are of the form $\boxed{d_{i}d_{j}d_{k}f}$ with $d_{i}d_{j}d_{k} = d_{0}d_{1}d_{2}$, and if we can find a filler, then $\lra{f}$ survives to $E\sp{\infty}\sb{2,p}$.

By definition, $C_{2}(\xd)$ is the limit of a diagram $\cB(_{+}\Delta_{0, 2}\op) \rightarrow X$, whose $1$-skeleton is
\myudiag[\label{tctwo}]{
x_{2} \ar@<2ex>[r]^{0} \ar[r]^{d_{1}} \ar@<-2ex>[r]^{d_{2}} & x_{1} \ar@<1ex>[r]^{d_{0}} \ar@<-1ex>[r]^{d_{1}} & x_{0}~.
}
Again, the $2$-simplices are either `coherences' for $\xd$ or witnesses for composition with
zero). By cofinality (\cite[Definition 4.1.3.1]{Lurie}), we can also write $C_{2}(\xd)$ as the limit of a diagram $\bar{D}\colon\cB(\pbDel_{-1, 2}\op ) \rightarrow X$ extending \eqref{tctwo} to the right by the augmentation $x_{0} \rightarrow\|\xd\|$.

The $1$-simplex $f : \Sigma^{p}y \rightarrow C_{2}(\xd)$ yields a $1$-simplex of $X_{/\bar{D}}$ and hence a map $\Delta^{1}\ast\cB(\pbDel_{-1, 2}\op ) \rightarrow X$.

The image of this map includes fillers for the $4$-simplices of the form $\boxed{d_{i}d_{j}d_{k}f}, k > 0$, so we are left with the $4$-simplices $d_{0}d_{1}d_{0}f, d_{0}d_{0}d_{0}f$, having two faces in common.
Thus (as in Example \ref{exam6.5}), we have fillers for all of the $4$-simplices except for
$\boxed{d_{0}d_{1}d_{0}f}$ and $\boxed{d_{0}d_{0}d_{0}f}$.
The $3$-dimensional facets of these two $4$-simplices are depicted in Figure \ref{fqcthrper}
(with the two common facets in the center).

The arrows indicate gluing of $3$-simplices along faces,
and the shaded $2$-simplices represent witnesses of composition with $0$. The $0$-facets
of the two (in the top right and top left)correspond to fillers of the two $3$-simplices we encountered in \S \ref{exam6.5}. The bottom $3$-simplices (labelled $H_{0, 0, 0}$  and
$H_{0, 0, 1}$) are coherence witnesses for $\xd$.

\stepcounter{theorem}
\begin{figure}[htbp]
\begin{center}
\begin{tikzpicture}
               \node at (1.5, 1.3) (1vert3) {$4$};
               \node at (3, 0) (1vert2) {$3$};
                 \node at (0, 0) (1vert1) {$2$};
               \node at (1.5, 2.6) (1vert0) {$1$};
\node at (1.5, -2.3) (2vert3) {$4$};
  \node at (3, -1) (2vert2) {$3$};
                 \node at (0, -1) (2vert1) {$2$};
               \node at (1.5, -3.6) (2vert0) {$0$};

  \node at (2, -4) (3vert3) {$3$};
                 \node at (3.5, -6.6) (3vert2) {$2$};
               \node at (0.5, -6.6) (3vert1) {$1$};
               \node at (2, -5.3) (3vert0) {$0$};
 \node at (5.5, 1.6) (4vert3) {$4$};
                \node at (7, -1) (4vert2) {$3$};
                 \node at (4, -1) (4vert1) {$1$};
               \node at (5.5, 0.3) (4vert0) {$0$};

 \node at (5.5, -2.4) (5vert3) {$4$};
                \node at (7, -5) (5vert2) {$2$};
                 \node at (4, -5) (5vert1) {$1$};
               \node at (5.5, -3.7) (5vert0) {$0$};

                 \node at (9.5, 1.3) (6vert3) {$4$};
                 \node at (8, 0) (6vert2) {$3$};
                   \node at (11, 0) (6vert1) {$2$};
               \node at (9.5, 2.6) (6vert0) {$1$};
\node at (9.5, -2.3) (7vert3) {$4$};
                 \node at (8, -1) (7vert2) {$3$};
                   \node at (11, -1) (7vert1) {$2$};
               \node at (9.5, -3.6) (7vert0) {$0$};

  \node at (10, -4) (8vert3) {$3$};
                 \node at (11.5, -6.6) (8vert2) {$2$};
               \node at (8.5, -6.6) (8vert1) {$1$};
               \node at (10, -5.3) (8vert0) {$0$};


                \draw (1vert3) -- (1vert2) node [midway, below] {\text{\tiny$f$}};
                \draw (1vert3) -- (1vert1) node [midway, sloped, below] {\text{\tiny$*$}};
                \draw (1vert3) -- (1vert0) node [midway, sloped, below] {\text{\tiny$*$}};
                \draw (1vert2) -- (1vert1) node [midway, sloped, below] {\text{\tiny$d_{0}$}};
                \draw (1vert2) -- (1vert0)  node [midway, sloped, above] {\text{\tiny$d_{0}d_{1}$}};
                \draw (1vert1) -- (1vert0) node [midway, sloped, above] {\text{\tiny$d_{0}$}};

                \draw (2vert3) -- (2vert2) node [midway, above] {\text{\tiny$f$}};
                \draw (2vert3) -- (2vert1) node [midway, sloped, below] {\text{\tiny$*$}};
                \draw (2vert3) -- (2vert0) node [midway, sloped, below] {\text{\tiny$*$}};
                \draw (2vert2) -- (2vert1) node [midway, above] {\text{\tiny$d_{0}$}};
                \draw (2vert2) -- (2vert0) ;
                \draw (2vert1) -- (2vert0) node [midway, sloped, below] {\text{\tiny$d_{0}d_{1}$}};

                \draw (3vert3) -- (3vert2) node [midway, sloped, above] {\text{\tiny$d_{0}$}};
                \draw (3vert3) -- (3vert1) node [midway, sloped, above] {\text{\tiny$d_{0}d_{1}$}};
                \draw (3vert3) -- (3vert0);
                \draw (3vert2) -- (3vert1) node [midway, sloped, below] {\text{\tiny$d_{0}$}};
                \draw (3vert2) -- (3vert0) node [midway, sloped, below] {\text{\tiny$d_{0}d_{1}$}};
                \draw (3vert1) -- (3vert0) node [midway, sloped, below] {\text{\tiny$d_{0}$}};

                \draw (4vert3) -- (4vert2) node [midway, above] {\text{\tiny$f$}};
                \draw (4vert3) -- (4vert1) node [midway, sloped, above] {\text{\tiny$*$}};
                \draw (4vert3) -- (4vert0) ;
                \draw (4vert2) -- (4vert1)  node [midway, sloped, below] {\text{\tiny$d_{0}d_{1}$}};
                \draw (4vert2) -- (4vert0);
                \draw (4vert1) -- (4vert0)  node [midway, sloped, above] {\text{\tiny$d_{0}$}};

                \draw (5vert3) -- (5vert2) node [midway, above] {\text{\tiny$*$}};
                \draw (5vert3) -- (5vert1) node [midway, above] {\text{\tiny$*$}};
                \draw (5vert3) -- (5vert0);
                \draw (5vert2) -- (5vert1) node [midway, sloped, below] {\text{\tiny$d_{0}$}};
                \draw (5vert2) -- (5vert0) ;
                \draw (5vert1) -- (5vert0) node [midway, below] {\text{\tiny$d_{0}$}};

                \draw (6vert3) -- (6vert2) node [midway, below] {\text{\tiny$f$}};
                \draw (6vert3) -- (6vert1) node [midway, sloped, below] {\text{\tiny$*$}};
                \draw (6vert3) -- (6vert0) node [midway, sloped, below] {\text{\tiny$*$}};
                \draw (6vert2) -- (6vert1) node [midway, sloped, below] {\text{\tiny$d_{1}$}};
                \draw (6vert2) -- (6vert0) node [midway, sloped, above] {\text{\tiny$d_{0}d_{1}$}};
                \draw (6vert1) -- (6vert0);

                \draw (7vert3) -- (7vert2) node [midway, below] {\text{\tiny$f$}};
                \draw (7vert3) -- (7vert1) node [midway, sloped, above] {\text{\tiny$*$}};
                \draw (7vert3) -- (7vert0) node [midway, sloped, above] {\text{\tiny$*$}};;
                \draw (7vert2) -- (7vert1) node [midway, sloped, above] {\text{\tiny$d_{1}$}};
                \draw (7vert2) -- (7vert0) node [midway, sloped, below] {\text{\tiny$d_{0}d_{1}d_{2}$}};
                \draw (7vert1) -- (7vert0);

                \draw (8vert3) -- (8vert2) node [midway, above] {\text{\tiny$d_{1}$}};
                \draw (8vert3) -- (8vert1) node [midway, sloped, above] {\text{\tiny$d_{0}d_{1}$}};
                \draw (8vert3) -- (8vert0);
                \draw (8vert2) -- (8vert1) node [midway, below] {\text{\tiny$d_{0}$}};
                \draw (8vert2) -- (8vert0) node [midway, sloped, below] {\text{\tiny$d_{0}d_{1}$}};
                \draw (8vert1) -- (8vert0) node [midway, above] {\text{\tiny$d_{0}$}};

\node at (2.0, 4) {$\boxed{d_{0}d_{0}d_{0}f}$};
        \node at (10.0, 4.0) {$\boxed{d_{0}d_{1}d_{0}f}$};
        \node at (1.5, 0.5) {\text{\tiny$G_{0}$}};
        \node at (1.5, -1.5) {\text{\tiny$G_{0}$}};
        \node at (9.5, 0.5) {\text{\tiny$G_{1}$}};
        \node at (9.5, -1.5) {\text{\tiny$G_{1}$}};
        \node at (2.0, -7.5) {$H_{0,0,0}$};
        \node at (10.0, -7.5) {$H_{0,0,1}$};
        \node at (5.5, -0.6) {\text{\tiny$H_{0,00}$}};
        \node at (5.5, -4.6) {\text{\tiny$H_{0,0}$}};
        \draw [opacity = 0.3, pattern=north west lines, pattern color=black] (1vert3.center) to (1vert1.center) to (1vert0.center);
        \draw [opacity = 0.3, pattern=north west lines, pattern color=black] (2vert3.center) to (2vert1.center) to (2vert0.center);
        \draw [opacity = 0.3, pattern=north west lines, pattern color=black] (4vert3.center) to (4vert1.center) to (4vert0.center);
        \draw [opacity = 0.3, pattern=north west lines, pattern color=black] (5vert3.center) to (5vert1.center) to (5vert0.center);
        \draw [opacity = 0.3, pattern=north west lines, pattern color=black] (5vert3.center) to (5vert2.center) to (5vert0.center);
        \draw [opacity = 0.3, pattern=north west lines, pattern color=black] (6vert3.center) to (6vert1.center) to (6vert0.center);
        \draw [opacity = 0.3, pattern=north west lines, pattern color=black] (7vert3.center) to (7vert1.center) to (7vert0.center);


        \draw [<->, dotted, thick] (2, 0.5)  to (2, -1.5);
        \draw [<->, dotted, thick] (3vert0)  to [out=40, in=30] (2.5, -2.5);
        \node at (3.5, -3) {\text{\tiny$\bottom$}};
        \draw [<->, dotted, thick] (1vert3)  to [out=40, in=120] (4.5, 0.5);
        \node [rotate=-40] at (4.3, 1.4) {\text{\tiny$\bottom$}};
        \draw [<->, dotted, thick] (4vert0)  to [out=130, in=120] (5vert0);
                \draw [<->, dotted, thick] (6.5, 0.5)  to [out=0, in=120] (6vert3);
                \node [rotate=30] at (7.3, 1) {\text{\tiny$\bottom$}};
                \draw [<->, dotted, thick] (10, 0.5)  to (10, -1.5);
                \draw [<->, dotted, thick] (5.0, -4.6)  to [out=-90, in=-80] (9.5, -6);
                \draw [<->, dotted, thick] (8vert0)  to [out=60, in=-60] (10.3, -2.8);
                \node  at (11, -3.2) {\text{\tiny$\bottom$}};

        \end{tikzpicture}
\end{center}
\caption{Partial quasi-category version of the $3$-permutahedron}
\label{fqcthrper}
\end{figure}
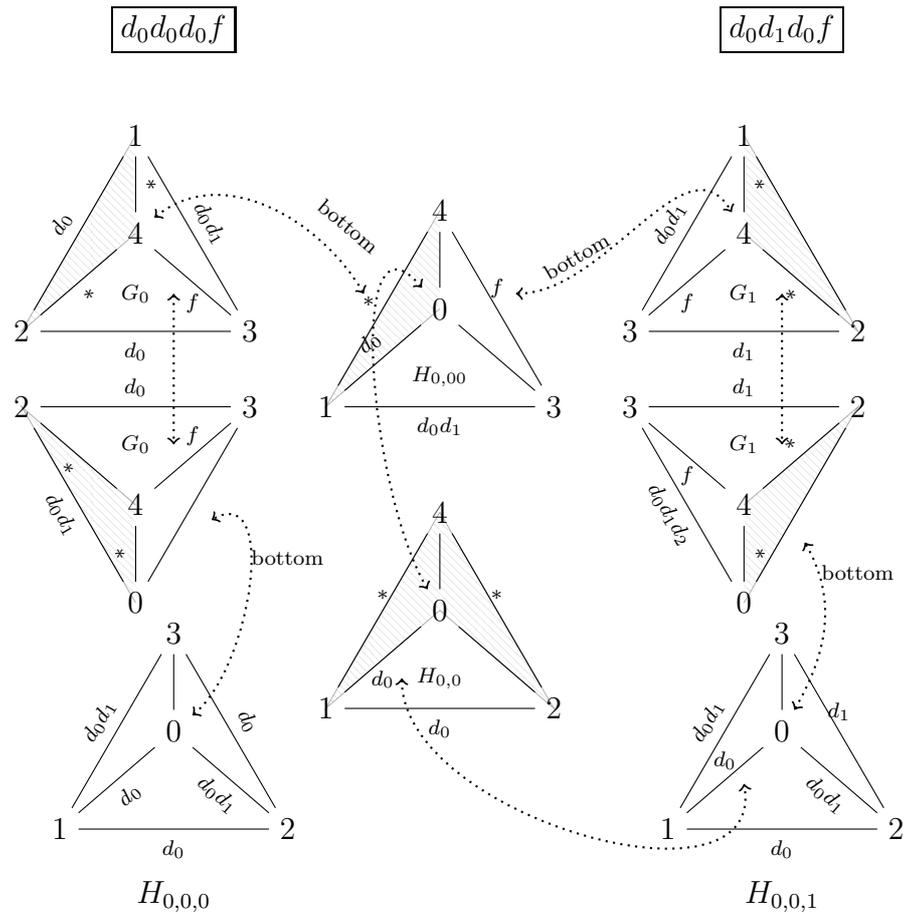
\end{example}

\newpage

\begin{remark}\label{rmk6.7}
  For a coherent simplicial object in a quasi-category, an obstruction to finding a lift as in Proposition \ref{thm6.3} must be expressed in terms of $r+1$-dimensional simplices. These correspond in the simplicially enriched case to vertices of the permutahedron $\mP^{r}$ \ -- \
in other words, the full diagram of Figure \ref{fqctwoper} is actually the dual of the permutahedron $\mP^{2}$. The edges of the permutahedron express the fact that we glue two of these $r+1$-dimensional simplices together along a face, and so on. It is therefore difficult to visualize the obstruction to finding a lift for \eqref{diagnk} when $r$ is large.
\end{remark}

%
%
\sect{The homotopy spectral sequence of a cosimplicial object}
\label{cdhssco}

In this section, we consider the dual case of the homotopy spectral sequence of a
cosimplicial object in an $\infty$-category (see \cite[X, \S 6]{BK}), and
briefly sketch the construction of the differentials, mainly in order to point out the differences with the simplicial case described in Section \ref{cdvho}.

By \cite[X, \S 5]{BK}, the Reedy model structure on cosimplicial spaces
$\sSet^{\Delta}$ is simplicial, with simplicial enrichment denoted by
$\Map_{\sSet^{\Delta}}(-, -)$.
We write $\Tot(\xu) =\Map_{\sSet^{\Delta}}(\Delta, \xu)$, where $\Delta$ is the cosimplicial space $[\bn] \mapsto \Delta^{n}$, and set  $\Tot^{n}(\xu):= \Map_{\sSet^{\Delta}}(\sk_{n}\Delta,\xu)$.

\begin{defn}\label{con7.1}
Suppose that $\xu\colon\Delta \rightarrow \sSet$ is a pointed, Reedy fibrant cosimplicial space such that each $x^{n}$ is an abelian homotopy group object in the category
of pointed spaces (so $\pi_{i}(x^{n})$ is an abelian group for all $i$). For each
$n \in \NN$, the $n$-th \emph{normalized cochain object} of $\xu$ is
\begin{myeq}\label{normcoch}
N^{n}(\xu)~:=~\bigcap_{i=0}^{n-1}\Ker(s^{i}\colon x^{n} \rightarrow x^{n-1})~.
\end{myeq}
We have fibration sequences
$$
\Omega^{n} N^{n}(\xu)~\xrightarrow{i_{n}}~\Tot^{n}(\xu)~\xrightarrow{p_{n}}~\Tot^{n-1}(\xu).
$$
The long exact sequence of these fibrations sequences gives rise to an exact couple of groups (by the assumption that each object is an abelian group object)
$$
\xymatrix
{
  \oplus_{p, n \in\NN} \pi_{p+n}\Tot^{n}(\xu) \ar[rr]^{p_{n}}   && \oplus_{n \in\NN} \pi_{p+n}\Tot^{n}(\xu) \ar[dl]^{\delta} \\
& \oplus_{p, n \in \NN} \pi_{p}\Omega^{n}N^{n}(\xu) \ar[ul]^{i_{n}}  &
}
$$
The spectral sequence associated to the exact couple
$$
E_{1}\sp{n,p} = \pi_{p}\Omega^{n}N^{n}(\xu) \implies \pi_{p}\Tot(\xu)
$$
is called the \emph{homotopy spectral sequence} of $\xu$ (see \cite{BK} or \cite[\S 8.1]{GJ2}). As before, to distinguish the differentials from (co)face maps, we denote the former by $\fd_{r}$.
\end{defn}

\begin{defn}\label{def7.2}
Suppose that $\mX$ is a fibrant simplicial category, $y \in \Ob(\mX)$ is a cogroup object,
and $\xu\colon\DK(\Delta) \rightarrow \mX$ is a coherent cosimplicial object.
Rectifying the homotopy coherent diagram $[\bn] \mapsto \Map_{\mX}(y, x^{n})$
(pointed by the zero maps, as in \S \ref{con3.3}) to a cosimplicial space $\wzu$, we may then choose a Reedy fibrant replacement $\zu$, and we call the homotopy spectral sequence for $\zu$ the \emph{homotopy spectral sequence} in $\mX$, associated to $\xu$ and $y$.
\end{defn}

\begin{construction}\label{con7.3}
For $y$ and $\xu$ as in \S \ref{def7.2}, suppose for simplicity that the cosimplicial space
$[\bn] \mapsto \Map_{\mX}(y, x^{n})$ is Reedy fibrant (in particular strict), and that $y$ itself is a suspension.

Assume given an element $\lra{f}$ in $E_{1}\sp{n,p}$ of the homotopy spectral sequence. Since $\mX(y, -)$ commutes with (homotopy) limits, we have
$$
\Omega^{n}N^{n}\Map_{\mM}(y, \xu) = \cap_{i=0}^{n-1} \Omega^{n}\Ker(\Map_{\mX}(y, x^{n}) \xrightarrow{\mM(-, s^{i})} \Map_{\mX}(y, x^{n-1}))
$$
Thus, we can represent $\lra{f}$ by a map $f\colon\Sigma^{p}y \otimes \Delta^{n} \rightarrow x^{n}$ with $s^{i}(f) = 0$ for $i < n$. Note that $f$ and  the maps
$0\colon\Sigma^{p}y \otimes \Delta^{k}\to x^{k}$ for $k < n$
together constitute an element of $\Tot^{n}$ of the cosimplicial space $\Map_{\mX}(y, \xu)$.

We write $\tr(\Delta^{n})$ for a triangulation of $\Delta^{n}$
given by the cone over $\partial \Delta^{n}$ with cone point at
the barycenter $b$, having $n$-simplices
$\sigma\sb{0},\dotsc\sigma\sb{n}$ (each the cone over the
corresponding $(n-1)$-simplex of $\partial \Delta^{n}$). By
\cite[X, \S 7]{BK} we have $\fd\sb{1}\lra{f} = \sum_{i=0}^{n}
(-1)\sp{i}d^{i} \circ f$, so we can represent it by a map
$\Sigma^{p-1}y \otimes\tr(\Delta^{n+1}) \rightarrow x^{n+1}$ which
restricts to $d^{i} \circ f$ on $\Sigma^{p-1}y\otimes\sigma\sb{i}$
(after identifying the $n$-simplex $\sigma\sb{i}$ with the
suspension of an $(n-1)$-simplex by sending the face opposite the
cone point $b$ to $0$). See Example \ref{examnone}.

For any $r\geq 2$, we may construct $\fd\sb{r}\lra{f}$ inductively, by assuming that
$\fd\sb{r-1}\lra{f}$ is represented by a map
$g\sp{r-1}\colon\Sigma\sp{p}y\otimes\Delta\sp{n+r-1}\to N\sp{n+r-1}\xu$ which restricts to $0$ on
$\Sigma\sp{p}y\otimes\partial\Delta\sp{n+r-1}$. In order for $\fd\sb{r}\lra{f}$ to be defined, we must have a nullhomotopy
$G\sp{r-1}\colon\Sigma\sp{p}y\otimes\Delta\sp{n+r}\to N\sp{n+r-1}\xu$ which is $0$ on all facets of $\Delta\sp{n+r}$ but the $0$-facet. Post-composing with the coface maps
$d\sp{i}\colon x\sp{n+r-1}\to x\sp{n+r}$ yields $n+r$ maps $d\sp{i}\circ G\sp{r-1}$
from all but the first of the $n+r$-simplices  $\sigma\sb{0},\dotsc\sigma\sb{n+r}$ of the triangulation $\tr\Delta\sp{n+r}$,
as before (with the facet $\sigma\sb{0}$ opposite the cone point $b$ mapping to $0$).
Since the simplices $\sigma\sb{i}$ are glued to each other in $\tr\Delta\sp{n+r}$ along
the base of the original $\Delta\sp{n+r}=C\Delta\sp{n+r-1}$, the boundary of
$\tr\Delta\sp{n+r}$ maps to $0$, so this defines a map
$\tr\Delta\sp{n+r}/\partial\tr\Delta\sp{n+r}\to\Map_{\mX}(y, N\sp{n+r}\xu)$.
representing $\fd\sb{r}\lra{f}$ in $\pi\sb{n+r}\Map_{\mX}(\Sigma\sp{p}y, N\sp{n+r}\xu)$.
\end{construction}

\begin{example}\label{examnone}
If $n=1$, then $\fd_{1}\lra{f}$ may be depicted by:
\mydiagram[\label{eqhexagon}]{
&& 0 && \\
\ar@{}[rrdr]_{\boxed{\text{\tiny$d^{2}f$}}} && &&  \ar@{}[lldl]^{\boxed{\text{\tiny$d^{0}f$}}} \\
&& \ar@{-}[uu]_{0} 0 && \\
0\ar@{-}[urr]^{0} \ar@{-}[uuurr]^{0} \ar[rrrr]_{0} && \ar@{}[urr]^>>>>>>>>>>>>>>>>>>>>>{\boxed{\text{\tiny$d^{1}f$}}} && \ar@{-}[ull]_{0} \ar@{-}[uuull]_{0} 0
}
where the maps $d\sp{i}f:\Sigma(\Sigma\sp{p-1}y)\otimes\Delta^{1}\to x\sp{2}$ are thought of
as mapping a $2$-simplex, rather than $\Delta^{1}\otimes\Delta^{1}$, to $\Map_{\mX}(\Sigma\sp{p-1}y, x^{2})$. In both versions the boundary maps to $0$, so
each of the $2$-simplices, as well as their alternating sum depicted in \eqref{eqhexagon}, represent elements in
$$
\pi\sb{2}\Map_{\mX}(\Sigma\sp{p-1}y, N\sp{2}\xu)~\cong~\pi\sb{p+1}\Map_{\mX}(y, N\sp{2}\xu).
$$
\end{example}

\begin{remark}\label{exam7.4}
Suppose that we now work with an arbitrary cosimplicial diagram $\xu$, with $f$
representing $\lra{f}\in E\sb{1}\sp{n,p} = \pi_{0}\Map_{\mX}(\Sigma^{p}y,\Omega^{n}N^{n}(\xu))$
(where $N^{n}(\xu)$ is now given by the homotopy limit corresponding to
\eqref{normcoch} \ -- \ compare \S \ref{csssec} and \S \ref{con2.6}).

Using \cite[X, Prop.\ 6.3ii]{BK}, we may represent $\lra{f}$ by a compatible sequence
of maps $f^{k}\colon\Sigma^{p}y\otimes \Delta^{k} \rightarrow x^{k}$ for $0\leq k\le n$, representing the image of $f$ in $E_{1}\sp{k,p}$. Note that if this sequence exists, it
is uniquely determined by $f=f\sp{n}$.

By comparing to the Reedy fibrant case, we see that we must have nullhomotopies
$F^{k}\colon f^{k} \sim 0$ for $k < n$, compatible in the sense that
$$
\xymatrix
{
C(\Sigma^{p}y \otimes \Delta^{k+1}) \ar[r]_>>>>>{F^{k+1}} & x^{k+1} \\
\ar[u]^{(\Id \otimes d^{i})} C(\Sigma^{p}y \otimes \Delta^{k}) \ar[r]_>>>>>>>>{F^{k}} & x^{k} \ar[u]_{d^{i}}
}
$$
commutes for all $0\leq i\leq k$.

The maps $d^{i} \circ F^{n-1}$ and $f^{n}$ thus glue together to yield a map
\begin{myeq}\label{cooperation}
\hat{F}^{n}\colon\Sigma^{p}y \otimes (\partial \Delta^{n} \times \Delta^{1} \cup_{\partial \Delta^{n} \times \{ 0\}} \Delta^{n} \times \{0\}) \rightarrow x^{n}.
\end{myeq}
We can think of this simplicial set in the domain as a `fattened up version' of the $n$-simplex, with the boundary mapping to zero. This represents $\lra{f}$ in $\pi_{n}\Map_{\mX}(\Sigma^{p}y, x^{n})$, and by comparison with \S \ref{con7.3}
we see that the value of differential $\fd\sb{1}$ at $\lra{f}$ is given by a map
$\Sigma^{p-1}y \otimes\tr(\Delta^{n+1}) \rightarrow x^{n+1}$. The higher differentials are
defined similarly.
\end{remark}

\begin{example}\label{examntwo}
For $n = 2$ we have:
\mydiagram[\label{eqfattet}]{
&& 0 \ar@/^1pc/[rrdddd]^{0} \ar@/_1pc/[lldddd]_{0} && \\
&& \ar@{-}[u] && \\
\ar@{}[rr]_{d^{2}F^{1}} && && \ar@{}[ll]^{d^{0}F^{1}} \\
& \ar@{-}[rr] \ar@{-}[uur] \ar@{}[ur]_>>>{\boxed{f^{2}}} &  & \ar@{-}[dr] \ar@{-}[uul] & \\
0\ar@{-}[rrrr]_{0} \ar@{-}[ur] && \ar@{}[u]_{d^{1}F^{1}} && 0
}
\end{example}

\begin{prop}\label{thm7.5}
A map $f$ as in \S \ref{exam7.4} represents an element at the $E_{r}$-page of the spectral sequence if and only if we can find a sequence of compatible maps $g^{k}\colon\Sigma^{p}y \otimes \Delta^{k} \rightarrow x^{k}$, with $n+1 \le  k \le r + n$, and with $g^{n+1}$ a representative for $\fd_{1}\lra{f}$, as well as compatible nullhomotopies $F^{k}\colon g^{k} \sim 0$ for all $k \le n+r-1$.
\end{prop}

\begin{proof}
  We will only prove the statement for $r = 2$ in the Reedy fibrant case: by hypothesis, we have a nullhomotopy $F^{n+1}\colon C(\Delta^{n+1}) \rightarrow x^{n+1}$ of $g^{n+1}$. By the dual of Lemma \ref{lem4.1}, we can assume that $F^{n+1}$ factors through $\Omega^{n}N^{n+1}(\xu)$ and $\fd_{2}\lra{f}$ is represented by
  $(g')^{n+2}\colon\Sigma^{p-1}y \otimes \tr(\Delta^{n+2}) \rightarrow x^{n+2}$, so that the interior cone point is zero and the interior $n+1$-simplices are nullhomotopies of $d^{i} \circ g^{n+1}$. The $i$-th `outer face' gets identified with $d^{i} \circ g^{n+1}$. There is a map $\phi_{n+2}\colon\Delta^{n+2} \rightarrow\tr(\Delta^{n+2})$ that `forgets' the inner n+2-simplices; precomposition with $\phi_{n+2}$ gives the required map $g^{n+2}$. The general case is proved similarly, by induction.
\end{proof}


\begin{thebibliography}{DKHS}
%
\bibitem[A]{AndG}
D.W.~Anderson,
``A generalization of the Eilenberg-Moore spectral sequence'',\hsm
\textit{Bull.\ AMS} \textbf{78} (1972), pp.~784-786.
%
\bibitem[BBS1]{BBSenH}
S.~Basu, D.~Blanc, \& D.~Sen,
``Higher structure in the unstable Adams spectral sequence'',\hsm
\textit{Homotopy, Homology \& Applications}, \textbf{23} (2021), pp.~69-94.
%
\bibitem[BBS2]{BBSenT}
S.~Basu, D.~Blanc, \& D.~Sen,
``Note on Toda brackets'',\hsm
\textit{J.\ Homotopy \& Rel.\ Structures} \textbf{15} (2020),   pp.~495-510.
%
\bibitem[BBS3]{BBSenHS}
S.~Basu, D.~Blanc, \& D.~Sen,
``The higher structure of unstable homotopy groups'',\hsm
preprint, 2020 ({\tt arXiv:2003.08842}).
%
\bibitem[Ba]{BauG}
H.J.\ Baues,
``Geometry of loop spaces and the cobar construction'',\hsm
\textit{Mem.\ AMS} {\bf 230}, AMS, Providence, RI, 1980.
%
\bibitem[BCM]{BCMillU}
M.~Bendersky, E.B.~Curtis, \& H.R.~Miller,
``The unstable Adams spectral sequence for generalized homology'',\hsm
\textit{Topology} \textbf{17} (1978), pp.~229-148.
%
\bibitem[Be1]{Bergner1}
J.E.~Bergner,
``A model cate\-gory struc\-ture on the ca\-tegory of simp\-licial cate\-go\-ries'',\hsm
\textit{Trans.\ AMS} \textbf{359} (2007), pp.~2043-2058.
%
\bibitem[Be2]{BERGNER3}
J.E.~Bergner,
``Three models for the homotopy theory of homotopy theories'',
\textit{Topology} \textbf{46} (2007), pp.~397-436.
%
\bibitem[Be3]{Bergner-Arising}
J.E.~Bergner,
``Complete Segal spaces arising from simplicial categories",
\textit{Trans.\ AMS} \textbf{361} (2009), pp.~525-546.
%
\bibitem[Bl]{BlaA}
D.~Blanc,
``Algebraic invariants for homotopy types'',
\textit{Math.\ Proc.\ Camb.\ Phil.\ Soc.} \textbf{127} (1999), pp.\ 497-523.
%
\bibitem[BJT1]{BJTurHA}
D.~Blanc, M.W.~Johnson, \& J.M.~Turner,
``Higher homotopy operations and Andr\'{e}-Quillen cohomology'',\hsm
\textit{Adv.\ Math.} \textbf{230} (2012), pp.~777-817.
%
\bibitem[BJT2]{BJTurHH}
D.~Blanc, M.W.~Johnson, \& J.M.~Turner,
``Higher homotopy invariants for spaces and maps'',\hsm
\textit{Algebraic \& Geometric Topology}, to appear ({\tt arXiv:1911.08259}).
%
\bibitem[BM]{BMarkH}
D.~Blanc \& M.~ Markl,
``Higher homotopy operations'',
\textit{Math.\ Zeit.} \textbf{345} (2003), pp.~1-29.
%
\bibitem[BS]{BSenH}
D.~Blanc \& D.~Sen,
``Higher cohomology operations and $R$-completion'',
\textit{Algebraic \& Geometric Topology} \textbf{18} (2018), pp.~247-312.
%
\bibitem[BV]{BVogHI}
J.M.~Boardman \& R.M.~Vogt,
\textit{Homotopy Invariant Algebraic Structures on Topological Spaces},\hsm
Springer-\-Verlag \textit{Lecture Notes Math.} \textbf{347},
Berlin-\-New York, 1973.
%
\bibitem[B]{BousH}
A.K.~Bousfield,
``Homotopy Spectral Sequences and Obstructions'',\hsm
\textit{Israel J.\ Math.} \textbf{66} (1989), pp.~54-104.
%
\bibitem[6A]{BCKQRSclM}
A.K.~Bousfield, E.B.~Curtis, D.M.~Kan, D.G.~Quillen, D.L.~Rector, \&
J.W.~Schlesinger,
``The mod-$p$ lower central series and the Adams spectral sequence'',\hsm
\textit{Topology} \textbf{3} (1966), pp.~331-342.
%
\bibitem[BF]{BFrieH}
A.K.~Bousfield \& E.M.~Friedlander,
``Homotopy theory of $\Gamma$-spaces, spectra, and bisimplicial sets'',\hsm
in M.G.~Barratt \& M.E.~Mahowald, eds.,
\textit{Geometric Applications of Homotopy Theory, II}
Springer \textit{Lec.\ Notes Math.} \textbf{658}, Berlin-\-New York, 1978,
pp.~80-130.
%
\bibitem[BK1]{BK}
A.K.~Bousfield \& D.M.~Kan,
\textit{Homotopy Limits, Completions, and Localizations},\hsm
Springer \textit{Lec.\ Notes Math.} \textbf{304}, Berlin-\-New York, 1972.
%
\bibitem[BK2]{BKanS}
A.K.~Bousfield \& D.M.~Kan,
``The homotopy spectral sequence of a space with coefficients in a ring'',\hsm
\textit{Topology} \textbf{11} (1972), pp.~79-106.
%
\bibitem[BK2]{BKaSQ}
A.K.~Bousfield \& D.M.~Kan,
``A second quadrant homotopy spectral sequence",
\textit{Trans.\ AMS} \textbf{177} (1973), pp.~305-318.
%
\bibitem[CF]{CFranH}
J.D.~Christensen \& M.~Frankland,
``Higher Toda brackets and the Adams spectral sequence in
triangulated categories'',\hsm
\textit{Alg.\ Geom.\ Top.} \textbf{17} (2017), pp.~2687-2735.
%
\bibitem[C]{JCohDS}
J.M.~Cohen,
``The decomposition of stable homotopy'',\hsm
\textit{Ann.\ Math.\ (2)} \textbf{87} (1968), pp.\ 305-320.
%
\bibitem[D]{DreS}
A.W.M.~Dress,
``Zur Spectralsequenz von Faserung'',\hsm
\textit{Inv.\ Math.} \textbf{3} (1967), pp.~172-178.
%
\bibitem[DKHS]{DHKSmitM}
W.G.~Dwyer, D.M.~Kan, P.S.~Hirschhorn, \& J.H.~Smith,
\textit{Homotopy limit functors on model categories and homotopical categories},\hsm
Math.\ Surveys \& Monographs \textbf{113}, AMS, Providence, RI, 2004.
%
\bibitem[DKSm]{Dwyer-Kan-Smith}
W.G.~Dwyer, D.M.~Kan, \& J.H.~Smith,
``Homotopy commutative diagrams and their realizations'',\hsm
\textit{J.\ Pure Appl.\ Alg.} \textbf{57} (1989), pp.~5-24.
%
\bibitem[DKSt]{Dwyer-Kan-Stover}
W.G.~Dwyer, D.M.~Kan, \& C.R.~Stover,
``The bigraded homotopy groups $\pi\sb{i,j}X$ of a pointed simplicial
space'',\hsm
\textit{J.\ Pure Appl.\ Alg.} \textbf{103} (1995), pp.~167-188.
%
\bibitem[GJ]{GJ2}
P.G.~Goerss \& J.F.~Jardine,
\textit{Simplicial Homotopy Theory},\hsm
Progress in Math.\ \textbf{179}, Birkh\"{a}user, Basel-Boston, 1999.
%
\bibitem[H]{HovM}
M.A.~Hovey,
\textit{Model Categories},\hsm
Math.\ Surveys \& Monographs \textbf{63}, AMS, Providence, RI, 1998.
%
\bibitem[J1]{Joyal1}
A.~Joyal,
``Quasi-categories and Kan complexes'',\hsm
\textit{J.\ Pure Appl.\ Alg.} \textbf{175} (2002), pp.~22-38.
%
\bibitem[J2]{Joyal-quasi-cat}
A.~Joyal,
``The Theory of Quasi-Categories and its Applications'',\hsm
Preprint, 2008, \url{http://mat.uab.cat/~kock/crm/hocat/advanced-course/Quadern45-2.pdf}.
%
\bibitem[L1]{Lurie}
J.~Lurie,
\textit{Higher Topos Theory},\hsm
Ann.\ Math.\ Studies \textbf{170}, Princeton U.\ Press, Princeton, 2009.
%
\bibitem[L2]{Lurie2}
J.~Lurie,
\textit{Higher Algebra},\hsm
online book, available at \url{http://people.math.harvard.edu/~lurie/papers/HA.pdf}, 2017.
%
\bibitem[Ma]{MasE}
W.S.~Massey,
``Exact couples in algebraic topology, I,II'',\hsm
\textit{Ann.\ Math.} \textbf{56} (1952), pp.~363-396.
%
\bibitem[Mi]{MilgI}
R.J.~Milgram,
``Iterated loop spaces'',\hsm
\textit{Ann.\ Math.} \textbf{84} (1966), pp.~386-403.
%
\bibitem[Q1]{QuiH}
D.G.~Quillen,
\textit{Homotopical Algebra},\hsm
Springer-\-Verlag \textit{Lec.\ Notes Math.} \textbf{20}, Berlin-\-New York,
1963.
%
\bibitem[Q2]{QuiS}
D.G.~Quillen,
``Spectral sequences of a double semi-simplicial group'',\hsm
\textit{Topology} \textbf{5} (1966), pp.~155-156.
%
\bibitem[Rec]{RectS}
D.L.~Rector,
``Steenrod operations in the Eilenberg-Moore spectral sequence'',\hsm
\textit{Comm.\ Math.\ Helv.} \textbf{45} (1970), pp.~540-552.
%
\bibitem[Rez]{Rezk}
C.~Rezk,
``A model for the homotopy theory of homotopy theory'',\hsm
\textit{Trans.\ American Mathematical Society} \textbf{353} (2001),
pp.~973-1007.
%
\bibitem[Ri]{Riehl1}
E.~Riehl,
``On the Structure of Simplicial Categories Associated to Quasi-Categories'',\hsm
\textit{Math.\ Proc.\ Camb.\ Phil.\ Soc.} \textbf{150} (2011), pp.~489-504.
%
\bibitem[Se]{SegC}
G.~Segal,
``Classifying spaces and spectral sequences'',\hsm
\textit{Publ.\ Math.\ Inst.\ Haut.\ \'{E}t.\ Sci.} \textbf{34} (1968),
pp.~105-112.
%
\bibitem[Sh]{ShipA}
B.E.~Shipley,
``An algebraic model for rational $S\sp{1}$-equivariant stable homotopy theory'',\hsm
\textit{Quart.\ J.\ Math.} \textbf{53} (2002), pp.~87-110.
%
\bibitem[Si]{SimpsonBook}
C.~Simpson,
\textit{The Homotopy Theory of Higher Categories},
New Math.\ Monographs \textbf{19},  Cambridge U.\ Press, Cambridge, UK, 2012.
%
\bibitem[St]{StoV}
C.R.~Stover,
``A Van Kampen spectral sequence for higher homotopy groups'',\hsm
\textit{Topology} \textbf{29} (1990), pp.~9-26.
%
\bibitem[T]{Toen}
B.~To{\"{e}}n,
``Vers une Axiomatisation de la Th{\'{e}}orie des Cat{\'{e}}gories Sup{\'{e}}rieures'',
\textit{$K$-Theory} \textbf{34} (2005), pp.~233-263.
%
%
\end{thebibliography}
\end{document}